\documentclass[letterpaper, 11pt]{article}

\usepackage{hyperref,enumerate}
\usepackage{amscd,amssymb,amsfonts,amsmath,amsthm}
\usepackage[all]{xy}

\makeindex

\newcommand{\bone}{\bf 1}
\newcommand{\Iso}{\operatorname{Iso}}
\newcommand{\pe}{\operatorname{pe}}
\newcommand{\Int}{\operatorname{Int}}
\newcommand{\dgk}{\operatorname{dgcat}_k}
\newcommand{\dgku}{\operatorname{dgcat}_{k[u,u^{-1}]}}
\newcommand{\ku}{k[u,u^{-1}]}

\newcommand{\Cb}{\mathbb{C\/}}
\newcommand{\C}{\operatorname{C\/}}
\newcommand{\Z}{\mathbb{Z\/}}
\newcommand{\A}{\mathbb{A\/}}
\newcommand{\X}{\mathcal{X\/}}
\newcommand{\Dstab}{\underline{\D^{\operatorname{b}}}(S)}
\newcommand{\Mod}{\underline{\operatorname{Mod\/}}}
\newcommand{\HP}{\operatorname{HP\/}}
\newcommand{\CC}{\operatorname{C\/}}

\newcommand{\HH}{\operatorname{HH\/}}
\newcommand{\MF}{\operatorname{MF\/}}
\newcommand{\MCM}{\operatorname{MCM\/}}
\newcommand{\id}{\operatorname{id\/}}
\newcommand{\RHom}{\operatorname{RHom\/}}
\newcommand{\IRHom}{\operatorname{R\underline{Hom}\/}}
\newcommand{\Hom}{\operatorname{Hom\/}}
\newcommand{\End}{\operatorname{End\/}}
\newcommand{\Ext}{\operatorname{Ext\/}}
\newcommand{\Tor}{\operatorname{Tor\/}}
\newcommand{\pd}{\operatorname{pd\/}}
\newcommand{\depth}{\operatorname{depth\/}}
\newcommand{\coker}{\operatorname{coker\/}}
\newcommand{\syz}{\operatorname{syz\/}}
\def\L{\operatorname{L\/}}
\def\H{\operatorname{H\/}}
\def\Spec{\operatorname{Spec\/}}
\def\stab{\operatorname{stab\/}}
\def\tr{\operatorname{tr\/}}

\def\D{{\operatorname{D\/}}}

\def\mod{\operatorname{mod\/}}
\newcommand{\mmod}[1]{#1 \text{-} \operatorname{mod\/}}
\def\op{\operatorname{op\/}}
\def\p{\mathfrak{p\/}}
\def\m{\mathfrak{m\/}}
\def\s{\mathfrak{s\/}}
\def\deli{\frac{\partial}{\partial\theta_i}}

\def\lra{\longrightarrow}
\newcommand{\Zt}{\mathbb{Z\/}/2}

\newcommand{\Ho}{\operatorname{Ho\/}}
\newcommand{\mfhom}{\operatorname{MF}}

\newtheorem{theo}{Theorem}[section]
\newtheorem{lem}[theo]{Lemma}
\newtheorem{cor}[theo]{Corollary}

\newtheorem{prop}[theo]{Proposition}

\theoremstyle{definition}
\newtheorem{defi}[theo]{Definition}

\theoremstyle{remark}

\newtheorem{ex}[theo]{Example}

\numberwithin{equation}{section}

\begin{document}

\title{Compact generators in categories of matrix factorizations}
\author{Tobias Dyckerhoff}
%\address{ University of Pennsylvania\\ Department of Mathematics\\  David Rittenhouse Laboratory
%3E6A\\  209 South 33rd Street\\ Philadelphia, PA 19104-6395}

%\email{tdyckerh@math.upenn.edu}

\maketitle

\begin{abstract} 
We study the category of matrix factorizations associated to the
germ of an isolated hypersurface singularity. This category is shown to admit
a compact generator which is given by the stabilization of the residue field. We
deduce a quasi-equivalence between 
the category of matrix factorizations and the dg derived category of an explicitly computable dg
algebra. Building on this result, we employ a variant of To\"en's derived Morita theory
to identify continuous functors between matrix factorization categories as
integral transforms. This enables us to calculate the Hochschild chain and
cochain complexes of these categories. Finally, we give
interpretations of the results of this work in terms of noncommutative geometry 
based on dg categories.
\end{abstract}

\tableofcontents

\section{Introduction}

Let $k$ be a field and let $(R,\m)$ be a regular local $k$-algebra of finite
Krull dimension with residue field $k$. 
We fix a non-zero element $w \in \m$ and introduce the corresponding hypersurface
algebra $S = R/w$. We are interested in the case when the hypersurface
$\Spec(S)$ has an isolated singularity at $\m$. Singularities of this kind have been a classical object of study for
centuries. The perspective on hypersurface singularities we take in this work is
one motivated by noncommutative algebraic geometry based on differential graded 
categories in the sense of \cite{pantev}. Namely, we study a dg category which we want to think of as
a category of complexes of sheaves on a hypothetical noncommutative space $\X$
attached to the singularity: the category of matrix factorizations of
$w$. In this work, we establish various properties of this category and
discuss the geometric implications for $\X$. Specifically, we show that $\X$
is a dg affine, homologically smooth and proper noncommutative Calabi-Yau space over $k$. 
We study the derived Morita theory of matrix factorization
categories which enables us to determine their Hochschild cohomology. We 
calculate the noncommutative analogues of Hodge and de Rham cohomology 
and show that the Hodge-to-de Rham spectral sequence degenerates.

A matrix factorization of $w$ is defined to be a $\Zt$-graded finite free
$R$-module $X$ together with an $R$-linear endomorphism $d$ of odd degree satisfying
$d^2 = w \id_X$. The collection of all matrix factorizations naturally forms a differential
$\Zt$-graded category which we denote by $\MF(R,w)$. The associated homotopy
category is denoted by $[\MF(R,w)]$.
Matrix factorizations first appeared in Eisenbud's work
\cite{eisenbud} on the homological algebra of complete intersections. Since
then, they have been used extensively in singularity theory. We refer the reader
to \cite{yoshino} for a survey as well as further references. In the unpublished work \cite{buchweitz}, Buchweitz introduced the
notion of the stabilized derived category, giving a new conceptual perspective on
Eisenbud's work and extending it to a more general context.
More recently, matrix factorizations were proposed by Kontsevich as descriptions
of $B$-branes in Landau-Ginzburg models in topological string theory.
As such they appear in the framework of mirror symmetry as for example explained
in \cite{seidel}.
Orlov \cite{orlov-2003} introduced the singularity category,
generalizing Buchweitz's categorical construction to a global setup, and
established various important results in
\cite{orlov-2005,orlov-2005-1,orlov-2009}.

In Section \ref{sect.survey},
we survey some important aspects of the inspiring articles
\cite{eisenbud} and \cite{buchweitz}, which lead to the intuitive
insight that the category of matrix factorizations describes the stable
homological features of the algebra $S$. The main purpose is to introduce notation
and to formulate the results in the form needed later on.

Section \ref{sect.conditions} introduces finiteness conditions on the singular
and critical locus of the function $w$, as well as finiteness and regularity
conditions on the ambient ring $R$ in both local and nonlocal scenarios.

Section \ref{sect.generator} addresses the question of the existence of
generators in matrix factorization categories. We construct a compact
generator of the category $\MF^{\infty}(R,w)$ consisting of factorizations of
possibly infinite rank. Our argument utilizes Bousfield localization
to reduce the problem to a statement
which we call Homological Nakayama Lemma for infinitely generated maximal
Cohen-Macaulay modules. This lemma seems to be an interesting result in its own
right since the
Nakayama lemma obviously fails for general infinitely generated modules. 
We use a method of Eisenbud to explicitly construct the generator as a matrix factorization 
corresponding to the stabilization of the residue field. Finally, we prove a localization theorem which allows us to generalize the generation result 
to certain nonlocal ambient rings.

Section \ref{sect.app} contains some first applications of the results on
compact generation.
We start by introducing a homotopy theoretic framework for $2$-periodic dg
categories, analogous to \cite{toen.morita}, that will serve as a natural
context to study matrix factorization categories.

Using a method due to Keller \cite{keller}, we obtain a quasi-equivalence between
the category $\MF^{\infty}(R,w)$ and the dg derived category of modules over a dg algebra $A$. 
This algebra $A$ is given as the endomorphism algebra of the compact generator. Our
concrete description of this generator as a stabilized residue field allows us to determine $A$
explicitly. As an immediate corollary, we obtain that the idempotent completion
of $[\MF(R,w)]$ coincides with $[\MF(\widehat{R},w)]$. We also give a short
proof of Kn\"orrer periodicity in this context.

In addition, we illustrate how to compute a 
minimal $A_{\infty}$-model for $A$. The transfer method we use originates from
the work of Gugenheim-Stasheff \cite{gug.stash} and Merkulov \cite{merkulov},
the elegant description in terms of trees is due to Kontsevich-Soibelman
\cite{kont.soib2}. 
In the case of a quadratic hypersurface the
$A_{\infty}$-structure turns out to be formal and we recover a variant of a
result of Buchweitz, Eisenbud and Herzog \cite{buchweitz2} describing
matrix factorizations as modules over a certain Clifford algebra.
In the general case, we are able to give partial formulas for the higher
multiplications which are neatly related to the higher coefficients of $w$.

In Section \ref{sect.morita}, we use To\"en's derived Morita theory for dg categories
\cite{toen.morita} to describe functors between categories of matrix
factorizations. It turns out that every continuous functor can be represented by an
integral transform. We describe the identity
functor as an integral transform with kernel given by the stabilized diagonal.
This allows us to calculate the Hochschild cochain complex of matrix
factorization categories as the derived endomorphism complex of the identity
functor.

Furthermore, we compute the Hochschild chain complex, proving along the way that 
it is quasi-isomorphic to the derived homomorphism complex between the inverse Serre
functor and the identity functor.

In the last section, we give interpretations of our results in terms of
noncommutative geometry in the sense of \cite{pantev}.

Finally, I would like to point out relations to previous work. I thank Daniel
Murfet for informing me that Corollary \ref{split} was first proven by Schoutens in
\cite{schoutens}. An independent proof by Murfet will be contained as an
appendix in
\cite{keller.bergh}. It is also possible to deduce the statement using results from
\cite{orlov-2009} as explained in \cite[11.1]{seidel}.
However, Theorem \ref{generator} is stronger since it also
implies that the category $\left[ \MF^{\infty}(R,w) \right]$ is compactly
generated which is essential to obtain Theorem \ref{equivalence}. We point
out that this fact can alternatively be obtained by using methods developed in \cite{chen.xw}.

I thank Paul Seidel for drawing my attention to his work \cite{seidel}. The
dg algebra $A = \End(k^{\stab})$ which we construct in Section
\ref{sect.app} already appears in Section 10 of loc. cit. and is interpreted as a
deformed Koszul dual. In fact, one may expect that the algebra $A$ is in fact
the Koszul dual, in the sense of \cite{positselski}, of the curved dg algebra $R$ with zero differential and
curvature $w$. Within this framework, Theorem \ref{equivalence} could be interpreted as an
equivalence of appropriate module categories over the curved algebra $R$ and its Koszul dual
$A$. The homological perturbation techniques which we apply in Section
\ref{sect.app}
were already used in Section 10 of \cite{seidel}.

The idea of describing functors between matrix factorization categories as
integral transforms appeared in \cite{khovanov}. Furthermore, our Theorem \ref{nakayama} is
inspired by Proposition 7 in loc. cit. 
As the authors informed me, the argument in loc. cit. 
is only valid for bounded below $\Z$-graded matrix factorizations of isolated singularities, which is sufficient for the purposes of loc. cit.
However, in this work we are specifically interested in the $\Zt$-graded case so
we have to use the alternative argument given in the proof of Theorem \ref{nakayama}.
As an application, we then also prove a $\Zt$-graded version of
\cite[Proposition 8]{khovanov} in the form of Corollary \ref{cohomology}.

The relation between idempotent completion and formal
completion is studied in a more general context on the level of
triangulated categories in \cite{orlov-2009}.

A heuristic calculation of the
Hochschild cohomology in the one-variable case with $w = x^n$ was carried out in
\cite{kapustin.rozansky}. This article already contains the essential idea to
represent the identity functor by a matrix factorization. 
We also mention that by \cite{takahashi-2005, 2005taksaito} the $\Z$-graded matrix
factorization category of a simple singularity is shown to be derived Morita
equivalent to the corresponding path algebra. This allows for a calculation of
the Hochschild invariants in this situation.

There are two alternative approaches to the calculation of Hochschild invariants of matrix
factorization categories. In \cite{segal}, the bar complex of the category is
used to calculate Hochschild homology. However, the author uses the product
total complex and it is not clear how this notion of homology is related to the
usual Hochschild homology. 
Therefore, some additional reasoning is required to make this argument work.
In forthcoming work by Caldararu-Tu \cite{caldararu}, the category of matrix
factorizations is considered as a category of modules over the above mentioned
curved dg algebra given by $R$ in even degree, zero differential and curvature
$w$. From this curved dg algebra the authors construct an explicit bar complex,
generalizing the bar complex for dg algebras. The Hochschild homology of the
curved dg algebra is then defined to be the cohomology of this
complex. The relation between the
Hochschild homology of the curved dg algebra and the Hochschild homology of the category of
modules over it is stated as a conjecture.

I should also mention that Theorem \ref{hochschild} has been anticipated
for some time. For example, it is stated without proof in
\cite{pantev}. However, to my knowledge, no complete proof has previously
appeared in the literature.\\

{\bf Acknowledgements.} I would like to thank my thesis advisor Tony Pantev for his
continuous support and interest in my work. Furthermore, I thank Matthew Ballard, Jonathan Block, Pranav Pandit,
Jim Stasheff and Bertrand To\"en for many inspiring conversations. I also thank
Mohammed Abouzaid, Vladimir Baranovsky, Chris Brav, Ragnar Buchweitz, Andrei Caldararu, Xiao-Wu Chen, Daniel Murfet, Ed
Segal and Paul Seidel for valuable comments. Finally, I am indebted to two
anonymous referees for their careful reading, corrections and suggestions.

\section{Homological algebra of matrix factorizations}
\label{sect.survey}

Let $(R,\m)$ be a regular local ring of finite Krull dimension. We fix a non-zero element $w \in \m$.

\begin{defi} The \emph{category of matrix factorizations} $\MF(R,w)$ of $w$ over $R$ is defined to be the
differential $\Zt$-graded category specified by the following data:
\begin{itemize}
\item The objects of $\MF(R,w)$ are pairs
$( X, d)$ where $X = X^0 \oplus X^1$ is a free $\Zt$-graded $R$-module of finite
rank equipped with an $R$-linear map $d$ of odd degree satisfying $d^2 = w
\id_X$.
\item The morphism complexes $\mfhom(X,X')$ are given by the $\Zt$-graded module of $R$-linear
maps from $X$ to $X'$ provided with the differential given by
\[
d(f) = d_{X'} \circ f - (-1)^{|f|} f \circ d_{X} \text{.}
\]
\end{itemize}
One easily verifies that $\mfhom(X,X')$ is a complex.
The \emph{homotopy category of matrix factorizations} $\left[ \MF(R,w) \right]$ is
obtained by applying $\H^0(-)$ to the morphism complexes of $\MF(R,w)$.
The set of morphisms in the homotopy category will be denoted by $[X,X']$.
\end{defi}
 
After choosing bases for $X^0$ and $X^1$, we obtain a pair
\[
\xymatrix{ 
X^1 \ar@<+.5ex>[r]^{\varphi} & \ar@<+.5ex>[l]^{\psi} X^0
}
\]
of matrices $(\varphi,\psi)$ such that 
\[
\varphi \circ \psi = \psi \circ \varphi = w \id \text{.}
\]
This immediately implies that the ranks of $X^0$ and $X^1$ agree, so $\varphi$
and $\psi$ are in fact square matrices.

\begin{ex}
Consider $R = \Cb[ [x]]$ and $w = x^n$. Then we have factorizations
\[
\xymatrix{ 
R \ar@<+.5ex>[r]^{x^k} & \ar@<+.5ex>[l]^{x^{n-k}} R 
}
\]
and these are in fact the only indecomposable objects in
$\operatorname{Z}^0(\MF(R,w))$ (cf. \cite{yoshino}, \cite{kapustin-li}). 
\end{ex}
\begin{ex}
Consider $R = \Cb[[x,y,z]]$ and $w = x^3 + y^3 + z^3 - 3xyz$. Suppose that
$(a,b,c) \in (\Cb^*)^3$ is a zero of $w$. The matrix
\[
\varphi =\left( \begin{array}{rrr} 
ax & cy & bz\\
cz & bx & ay\\
by & az & cx
\end{array} \right)
\]
satisfies $\det(\varphi) = abc w$. Thus setting $\psi =
\frac{1}{abc}\varphi^{\#}$, where $-^{\#}$ denotes the matrix of cofactors, we obtain a
family of rank $3$ factorizations parameterized by $\{ w = 0 \} \subset
(\Cb^*)^3$ (cf. \cite{brunner}).
\end{ex}

We will now explain the relevance of the category of matrix factorizations in
terms of homological algebra. To this end, we recall how matrix factorizations
naturally arise in Eisenbud's work \cite{eisenbud}. 

\subsection{Eisenbud's matrix factorizations}

\label{subsect.factorizations}

We start with a prelude on the homological algebra of regular local rings that will 
put us in the right context.
Let $(R,\m)$ be a regular local ring and let $M$ be a finitely generated
$R$-module. A sequence $x_1,\dots,x_r \in \m$ is called an $M$-sequence if $x_i$
is a nonzerodivisor in $M/(x_1,\dots,x_{i-1})$ for all $1 \le i \le r$. The
\emph{depth} of $M$ is the length of a maximal $M$-sequence. The \emph{projective
dimension} $\pd(M)$ of $M$ is the length of a minimal free resolution of $M$.
The Auslander-Buchsbaum formula (see e.g. \cite{eisenbud.comalg}) relates these notions via
\[
\pd(M) = \dim(R) - \depth(M) \text{.}
\]
This yields a rather precise understanding of free resolutions over regular
local rings. Let us point out two immediate important consequences. Firstly,
the length of minimal free resolutions is bounded by the Krull dimension of $R$.
Secondly, if the depth of a module $M$ equals the Krull dimension of $R$, then
$M$ is free.

A natural problem is to try and obtain a similar
understanding of free resolutions over singular rings. An example of such a 
ring is a hypersurface singularity defined by $S =
R/w$, where $w$ is singular at the maximal ideal. 
It turns out that in contrast to the regular case, the condition
\begin{equation}
\label{mcm}
\depth(M) = \dim(S)
\end{equation}
does not imply that $M$ is free. A finitely generated $S$-module satisfying 
(\ref{mcm}) is called a \emph{maximal Cohen-Macaulay module}.

Let $M$ be a maximal Cohen-Macaulay module over $S$. We may consider $M$ as an $R$-module
which is annihilated by $w$ and use the Auslander-Buchsbaum formula
\[
\pd_R(M) = \depth(R) - \depth(M)
\]
to deduce that $M$ admits an $R$-free resolution of length $1$. Hence, we obtain
an exact sequence
\[
\xymatrix{
0 \ar[r] & X^1 \ar[r]^{\varphi} &  X^0 \ar[r] &  M \to 0
}
\]
where $X^0$ and $X^1$ are free $R$-modules.
Since multiplication by $w$ annihilates $M$, there exists a homotopy $\psi$ such
that the diagram 
\[
\xymatrix{
X^1 \ar[r]^{\varphi} \ar[d]_w & \ar[dl]_{\psi} X^0 \ar[d]^w\\
X^1 \ar[r]^{\varphi} & X^0
}
\]
commutes. Thus, the pair $(\varphi, \psi)$ is a matrix factorization of $w$,
such that the original maximal Cohen-Macaulay module $M$ is isomorphic to
$\coker(\varphi)$.

We come back to the question about properties of $S$-free resolutions of $M$. Curiously,
every maximal Cohen-Macaulay module over $S$ admits a $2$-periodic $S$-free resolution.
It is obtained by reducing the corresponding matrix factorization modulo $w$
and extending $2$-periodically:
\[
\xymatrix{ 
 \dots \ar[r] & \overline{X^1} \ar[r]^{\overline{\varphi}} & 
\overline{X^0} \ar[r]^{\overline{\psi}} &  \overline{X^1}
\ar[r]^{\overline{\varphi}} &
\ar[r] \overline{X^0} & M \ar[r] & 0 \text{.}
}
\]

We illustrate the consequences of this construction from a categorical point of view. Consider 
the \emph{stable category of maximal Cohen-Macaulay $S$-modules} $\underline{\MCM}(S)$ which is defined as follows.
The objects are maximal Cohen-Macaulay modules, the morphisms are defined by
\[
\underline{\Hom}_S(M,M') = \Hom_S(M,M') / P \text{,}
\]
where $P$ denotes the set of $S$-linear homomorphisms factoring through some free $S$-module.
Reversing the above construction we can associate the maximal Cohen-Macaulay module $\coker(\varphi)$ 
to a matrix factorization given by 
\[
\xymatrix{ 
X^1 \ar@<+.5ex>[r]^{\varphi} & \ar@<+.5ex>[l]^{\psi} X^0 \text{.}
}
\]
In fact, this assignment extends to a functor
\[
\coker : \; \left[ \MF(R,w) \right] \to \underline{\MCM}(S)
\]
establishing an equivalence between the homotopy category of matrix factorizations
and the stable category of maximal Cohen-Macaulay modules.

\subsection{Buchweitz's stabilized derived category}

We now want to focus on resolutions of arbitrary finitely generated $S$-modules. 
To this end, we are lucky as it turns out that high enough syzygies of any such 
module are maximal Cohen-Macaulay. This follows from the homological
characterization of maximal Cohen-Macaulay modules as being
$\Hom_S(-,S)$-acyclic combined with the fact that $S$ has finite injective
dimension. An immediate implication is the following striking result.

\begin{theo}[Eisenbud]  \label{twoperiodic} Every finitely
generated $S$-module admits a free resolution which will eventually become
$2$-periodic. 
\end{theo}

In other words, the resolution ``stabilizes'' leading
to the slogan that the category $\underline{\MCM}(S)$ describes the stable
homological algebra of $S$.

While the category $\underline{\MCM}(S)$ restricts attention to maximal Cohen-Macaulay modules,
Buchweitz's stabilized derived category is designed to capture the fact
that \emph{arbitrary} finitely generated $S$-modules stabilize.
Let $\D^\text{b}(S)$ denote the derived category of all complexes of $S$-modules with
finitely generated total cohomology. Such a complex is called \emph{perfect} 
if it is isomorphic in $\D^\text{b}(S)$ to a bounded complex of free $S$-modules. The
full triangulated subcategory of $\D^\text{b}(S)$ formed by the perfect complexes is
denoted by $\D^\text{b}_{\text{perf}}(S)$. It is easy to see that $\D^\text{b}_{\text{perf}}(S)$
forms a thick subcategory of $\D^\text{b}(S)$. The \emph{stabilized derived category of $S$} is then
defined to be the Verdier quotient
\[
\Dstab :=   \D^\text{b}(S) / \D^\text{b}_{\text{perf}}(S) \text{.}
\]
There exists an obvious functor 
\[
\underline{\MCM}(S) \to \Dstab
\]
which Buchweitz proves to be an equivalence of categories. Observe that 
$\Dstab$ as well as $\left[ \MF(R,w) \right]$ are naturally triangulated
categories that via the above equivalences induce two triangulated structures on
$\underline{\MCM}(S)$. Those structures turn out to be isomorphic (via the
identity functor). The triangulated structure can also be constructed directly
using the fact that $\underline{\MCM}(S)$ is the stable category associated to
the Frobenius category $\MCM(S)$ (cf. \cite{keller-frob}).

It is interesting to describe the morphisms in the category 
$\Dstab$.

\begin{prop}[Buchweitz]
Let $X$, $Y$ be complexes in $\Dstab$. Then there exists a natural
number $i(X,Y)$ such that 
\[
\Hom_{\Dstab}(X,Y[i]) \cong \Hom_{\D^\text{b}(S)}(X,Y[i])
\]
for $i \ge i(X,Y)$.
\end{prop}

The proposition explains the nomenclature for $\Dstab$. The
$\Ext$-groups in the derived category ``stabilize'' in high degrees. After this
stabilization has taken place, the $\Ext$-groups and the
$\underline{\Ext}$-groups coincide. Buchweitz introduces $\Dstab$ more generally
for Gorenstein algebras. In our specific situation of a local hypersurface
algebra $S$ the phenomenon of stabilization just translates into the above mentioned fact that resolutions over $S$
eventually become $2$-periodic.

Combining the equivalences of categories explained in this subsection, we conclude
that a finite $S$-module interpreted as an object of $\Dstab$ functorially
corresponds to a matrix factorization. If $L$ is an $S$-module, we call the corresponding matrix
factorization $L^{\stab}$ the \emph{stabilization of} $L$. 
The objects of our interest tend to naturally arise as objects of $\Dstab$ and we will
analyze them computationally by studying their stabilization. 

\subsection{Stabilization}

\label{stabilization}

Let $L$ be an $S$-module. In \cite[Section 7]{eisenbud}, Eisenbud gives a method for explicitly constructing $L^{\stab}$
in terms of an $R$-free resolution of $L$. We apply his construction in
the case when $L$ is a module of the form $L = R/I$ such that the ideal $I$ is
generated by a regular sequence and $w \in I$. Since we use the construction
throughout the article, we
give a detailed description of this special case. We point out that this
construction already appears as a key technique in \cite[Section
2]{buchweitz.greuel}. Its meaning was further
clarified in \cite{avramov} via the use of dg modules over Koszul dg algebras.

Let $f_1, \dots, f_m$ be a regular sequence generating $I$. Consider the 
corresponding Koszul complex
\[
K = ({\textstyle\bigwedge\nolimits}^* V, \;s_0) \text{,}
\]
where $V = R^m$ and $s_0$ denotes contraction with $(f_1,\dots,f_m) \in
\Hom_R(V,R)$. The
complex $K$ is an $R$-free resolution of $L$.
Since $w$ annihilates the $R$-module $L$, multiplication by $w$ on $K$ is homotopic to zero. In
fact, we can explicitly construct a contracting homotopy. Since $w \in I$, we can write 
$w = \sum_i f_i w_i$ for some elements $w_i \in R$. 
Exterior multiplication with the element $(w_1,\ldots,w_m) \in V$
defines a contracting homotopy which we denote by $s_1$.

Since both $s_0$ and $s_1$ square to $0$, the $\Zt$-graded object
\[
(\bigoplus_{i=0}^m {\textstyle\bigwedge\nolimits}^{i} V, \;s_0 + s_1 )
\]
defines a matrix factorization of $w$. We claim that it represents the
stabilization of $L$. To see this we will construct an explicit $S$-free
resolution of $L$. Define $Z$ to be the total complex of the double complex
\[
\xymatrix{
\ddots & & \ddots  & &\ddots & \\
 & \overline{K^m} \ar[r]^{s_0} & \overline{K^{m-1}} \ar[r]^{s_0} \ar[d]^{s_1} &
 \overline{K^{m-2}} \ar[r] \ar[d]^{s_1} & \dots\ar[r] & \overline{K^{0}} \ar[d]^{s_1}\\
 &  & \overline{K^m} \ar[r]^{s_0} & \overline{K^{m-1}} \ar[d]^{s_1} \ar[r] &
 \dots\ar[r] & \overline{K^{1}} \ar[r]^{s_0}\ar[d]^{s_1} & \overline{K^{0}} \ar[d]^{s_1}\\
 &  & & \overline{K^m} \ar[r] & \dots \ar[r] & \overline{K^{2}} \ar[r]^{s_0} &
 \overline{K^{1}}
\ar[r]^{s_0} & \overline{K^{0}}
}
\]
where $\overline{-}$ denotes the functor $- \otimes_R S$. 

\begin{lem}[Eisenbud] \label{sfree}The complex $Z$ is an $S$-free resolution of $L$.
\end{lem}

\begin{proof}
We use the spectral
sequence arising from the horizontal filtration to compute the cohomology of
$Z$. On the first page we obtain
\[
\xymatrix{
\ddots & L \ar[d]^{g} \\
& L & L \ar[d]^{g} \\
& & L & L 
}
\]
since the complex $K \otimes_R S$ is isomorphic in $\D^\text{b}(R)$ to the
complex $L \otimes_R (R \stackrel{w}{\lra} R)$ and $L$ is annihilated by $w$.
To determine the map $g$ we use the roof establishing the just mentioned
isomorphism.
\[
\xymatrix@!0@C=4cm@R=1.7cm{
& \ar[dl]_{\simeq}^{p_1} \ar[dr]^{\simeq}_{p_2} (\dots K^2 \stackrel{s_0}{\lra} K^1 \stackrel{s_0}{\lra} K^0) \otimes_R (R \stackrel{w}{\lra} R)
\\
(\dots K^2 \stackrel{s_0}{\lra} K^1 \stackrel{s_0}{\lra} K^0) \otimes_R S & & L \otimes_R (R \stackrel{w}{\lra} R)
}
\]
We introduce the homotopy
\[
\xymatrix{
R \ar[r]^w & \ar@{.>}@/^/[l]^t R\\
}
\]
which is simply the identity map on $R$. Then the map $s_1 \otimes 1 + 1 \otimes
t$ is a map of degree $-1$ on the complex forming the apex of the roof. The map
induced on $K \otimes_R S$ via $p_1$ is $s_1 \otimes 1$ while the one
induced on $L \otimes_R (R \to R)$ via $p_2$ is $1 \otimes t$. This proves that the
vertical maps $g$ on the first page of the above spectral sequence are in fact
given by the identity on $L$. Passing to the second page of the spectral sequence we immediately obtain
the result.
\end{proof}

\begin{cor} \label{cor.stab} The stabilization $L^{\stab}$ of $L$ is given by the matrix factorization 
\[
(\bigoplus_{i=0}^m {\textstyle\bigwedge\nolimits}^{i} V, \;s_0 + s_1 ) \text{.}
\]
\end{cor}

\begin{proof}
	This simply follows by inspecting the explicit form of the constructed $S$-free
	resolution of $L$. It becomes $2$-periodic after $m$ steps where the
	$2$-periodic part is exactly the reduction modulo $w$ of the given
	matrix factorization.
\end{proof}

It is convenient to formulate this construction in the language of supergeometry as it is
often done in the physics literature.
The underlying space of $L^{\stab}$ can be interpreted as the superalgebra 
$R\left<\theta_1, \ldots, \theta_m\right>$ where $\theta_i$ are odd supercommuting variables 
and $R$ has degree $0$. The twisted differential defining the factorization corresponds to the odd differential
operator $\delta = \sum \delta_i$ with
\[
\delta_i = f_i \deli + w_i \theta_i \text{.}
\]
This interpretation is useful, since every $R$-linear endomorphism of the super polynomial ring
$R\left<\theta_1, \ldots, \theta_m\right>$ is represented by a differential operator, as can be easily seen by a dimension count. 
Thus, denoting the $\Z/2$-graded $R$-module of all polynomial differential
operators on $R\left<\theta_1, \ldots, \theta_m\right>$ by $A$, we obtain an explicit description 
of the dg algebra of endomorphisms of $L^{\stab}$ as
\[
\mfhom(L^{\stab}, L^{\stab}) \cong (A,\; [\delta, -]) 
\]
with $[\delta, \theta_i] = f_i$ and $[\delta, \frac{\partial}{\partial
\theta_i}]  = w_i$.

Finally, we point out that, ignoring the differential $\delta$, this setup is classical in the context of Clifford
algebras. Indeed, the graded
algebra $R\left<\theta_1, \ldots, \theta_m\right>$ is the exterior
algebra of the free $R$-module $B$ generated by the $\theta_i$ in odd degree. The
endomorphism algebra of this exterior algebra is well-known to be the
Clifford algebra of the associated hyperbolic form on $B \oplus B^*$ (cf.
\cite[IV, \S 2]{knus}). 
Denoting the elements in the basis of $B^*$ dual to $\{\theta_i\}$ by
$\{\frac{\partial}{\partial \theta_i}\}$, we arrive at our description of $A$ via
differential operators.

\subsection{Nonlocal hypersurfaces}
\label{nonlocal}

We remark that the theory described in this section generalizes to a
nonlocal situation where $R$ is a regular ring of finite Krull dimension.
For $w \in R$ we define $S=R/w$ and observe that the whole framework generalizes in the following
way. We define maximal Cohen-Macaulay modules to be $\Hom_S(-,S)$-acyclic
finitely generated $S$-modules. The matrix factorizations are defined in complete analogy to the
local case, with the exception that the two graded pieces of a factorization are
allowed to be finitely generated projective $R$-modules.
As already pointed out in \cite{buchweitz}, the important fact is that the
hypersurface algebra $S$ is of finite injective dimension over itself. Using  
the flatness of localization, we can generalize all statements in this section mutatis mutandis.

We use the same notations $\MF(R,w)$, $\MCM(S)$, etc. to also denote the
relevant categories and constructions in this more general situation.

\section{Singular and critical locus}
\label{sect.conditions}

This short section introduces various technical conditions which we will use
later on. 

\subsection{Local ambient ring}
\label{subsect.local}

We start by recalling a well-known result (see e.g. \cite[Proposition
(1.2)]{looijenga} for $k=\mathbb C$).

\begin{prop}\label{singeq}
	Let $k$ be a field of characteristic $0$. Consider a polynomial $w \in
	R = k[x_1, \dots, x_n]$. Let $Z$ be the scheme-theoretic zero locus of the $1$-form $dw$ on
	$\A^n$ (i.e. the critical locus of $w$) and denote the hypersurface
	algebra $R/w$ by $S$. Then the following are equivalent.
	\begin{enumerate}[(1)]
		\item\label{singeq.1} The hypersurface $\Spec(S)$ has 
			isolated singularities, i.e. $S_{\p}$ is regular
			for every non-maximal prime $\p \in \Spec(S)$.
		\item\label{singeq.2} The restriction of the critical locus $Z$
			to $\Spec(S)$ is a $0$-dimensional scheme.
		\item\label{singeq.3} For each singular point $\m \in \Spec(S)$,
			the restriction of the critical locus $Z$ to
			$\Spec(R_{\m})$ is a $0$-dimensional scheme.
	\end{enumerate}
\end{prop}

\begin{proof}
	The equivalence of (\ref{singeq.1}) and (\ref{singeq.2}) follows from the Jacobian
	Criterion (\cite[16.19]{eisenbud.comalg}). We are left to show that (\ref{singeq.2}) implies
	(\ref{singeq.3}). Assume (\ref{singeq.2}) holds but there exists a
	singular point $\m \in \Spec(S)$ such that the critical locus $Z$ does not 
	restrict to a $0$-dimensional subscheme of $\Spec(R_{\m})$. 
	Then there exists a non-maximal prime $\p \subset \m$ along which $dw$ vanishes. 
	We claim that $w$ must vanish along $\p$. Since $w$ vanishes on $\m$ it
	suffices to show that $w$ is constant along $\overline{\p}$. Indeed, if $w$ were
	non-constant it would map the irreducible variety $\overline{\p}$
	onto a dense subset $U$ of $\A^1$. But then, since $dw$ vanishes along
	$\overline{\p}$, the
	Jacobian Criterion would imply that every fiber of $w$ over $U$ is singular, contradicting generic
	smoothness (\cite[16.23]{eisenbud.comalg}). Hence, $w$ vanishes
	along $\p$. But then, again by the Jacobi Criterion, $S_{\p}$ is
	not regular which contradicts our assumption.
\end{proof}

For any pair $(R,w)$ given by a regular local ring $(R,\m, k)$ of Krull
dimension $n$ containing $k$ and a non-zero element $w \in \m$, we introduce the following
conditions. If there exists a sequence of $k$-linear derivations $\partial_1, \dots,
\partial_n$ of $R$ such that
\begin{equation}\label{is.sing}
\dim_k S/(\partial_1 w, \dots, \partial_n w) < \infty
\end{equation}
then we say $(R,w)$ has \emph{isolated singular locus}. If there exists a sequence of derivations $\partial_1, \dots,
\partial_n$ of $R$ such that
\begin{equation}\label{is.crit}
\dim_k R/(\partial_1 w, \dots, \partial_n w) < \infty
\end{equation}
then we say $(R,w)$ has \emph{isolated critical locus}.
Note that Proposition \ref{singeq} shows that these notions are
equivalent under the conditions stated. In general, this is not the case.

\subsection{Nonlocal ambient ring}
\label{subsect.nonlocal}

Let $k$ be a field, $R$ a $k$-algebra and $w \in R$. We will frequently refer to
the following condition on the pair $(R,w)$.

\begin{enumerate}[(A)]
	\item \label{a}
$R$ is essentially of finite type over $k$, equidimensional of dimension $n$, and the module of
K\"ahler differentials $\Omega_{R/k}$ is locally free of rank $n$.  
Further assume, that the restriction of the zero locus of $dw \in \Omega_{R/k}$ to $\Spec(R/w)$ is a
$0$-dimensional scheme supported on a unique closed point $\s$ of $\Spec(R/w)$
with residue field $k$.
\end{enumerate}

In particular, by \cite[II, 8.7]{hartshorne} the local ring $R_{\s}$ is regular,
and $(R_{\s},w)$ has isolated singular locus in the sense of (\ref{is.sing}).
Further, the $\s$-adic completion $(\widehat{R_{\s}},w)$ has isolated singular locus.
The main results of Sections \ref{sect.generator} and \ref{sect.app} will be
valid under the condition $(\ref{a})$.

In Section \ref{sect.morita}, we will employ the operation
\[
(R,w) \otimes_k (R',w') := (R \otimes_k R', w \otimes 1 + 1 \otimes w') \text{.}
\]
Unfortunately, condition (\ref{a}) is not necessarily preserved under this
operation and we therefore introduce the following stronger condition.

\begin{enumerate}[(A)]
	\setcounter{enumi}{1}
	\item\label{b}
$R$ is essentially of finite type over $k$, equidimensional of dimension $n$, and the module of
K\"ahler differentials $\Omega_{R/k}$ is locally free of rank $n$.  
Further assume, that the zero locus of $dw \in \Omega_{R/k}$ is a
$0$-dimensional scheme supported on a unique closed point $\s$ of $\Spec(R)$ with residue
field $k$ and $w \in \s$.
\end{enumerate}

In particular, $(R_{\s},w)$ has isolated critical locus in the sense of
(\ref{is.crit}) and the same is true for the $\s$-adic completion
$(\widehat{R_{\s}},w)$. Also note, that condition (\ref{b}) is preserved under
the above mentioned tensor product operation. 

\section{Generators in matrix factorization categories}
\label{sect.generator}

We fix a pair $(R,w)$ consisting of a regular local ring $(R,\m, k)$ of Krull
dimension $n$ containing $k$ and a non-zero element $w \in \m$. Throughout this section
we assume that $(R,w)$ has isolated singular locus (\ref{is.sing}). We
denote the hypersurface algebra $R/w$ by $S$.

In order to obtain a setup in which we can talk about compactness and
apply the technique of Bousfield localization (cf. \cite{neeman}), we are forced to enlarge the 
category of matrix factorizations to admit arbitrary coproducts. 
To this end, we use the category $\MF^\infty(R,w)$
of matrix factorizations of possibly infinite rank. By the existence of a
compact generator this category will simply turn out to be an explicit model for
the dg derived category of unbounded modules over $\MF(R,w)$ (see Theorem
\ref{equivalence}).

We introduce some general notions. Let $\mathcal{T}$ be a triangulated category
admitting infinite coproducts. Let $X$ be an object of $\mathcal{T}$. We call $X$
\emph{compact} if the functor $\Hom(X, -)$ commutes with infinite coproducts.
The object $X$ is a \emph{generator} of $\mathcal{T}$ if the smallest triangulated subcategory of $\mathcal{T}$
containing $X$ and closed under coproducts and isomorphisms is $\mathcal{T}$ itself. The full subcategory
of $\mathcal{T}$ consisting of all objects $Y$ satisfying $\Hom(X[i],Y) = 0$ for
all integers $i$ is called the \emph{right orthogonal complement of $X$}.
Using Bousfield localization, one shows that for a compact object $X$
the right orthogonal complement of $X$ is equivalent to $0$ if and only if $X$
is a generator of $\mathcal{T}$ (\cite[Lemma 2.2.1]{schwede}). 

Note that the full subcategory $[\MF(R,w)] \subset [\MF^\infty(R,w)]$ consists
of compact objects. This simply follows since both even and odd components of a
factorization $X$ in $\MF(R,w)$ are finitely generated $R$-modules. Thus, already on
the level of morphism complexes, $\MF(X, -)$ commutes with coproducts.
We can state the result which we prove in this section.

\begin{theo}
\label{generator}
Assume that $(R,w)$ has isolated singular locus (\ref{is.sing}) and consider the residue field $k$ 
as an $S$-module. Then $k^{\stab}$ is a compact generator of the
triangulated category $\left[ \MF^\infty(R,w) \right]$.
\end{theo}

For the proof we need some preparation.

\subsection{The stabilized residue field}

We start by computing the matrix
factorization $k^{\stab}$ and its endomorphism dg algebra explicitly, using
Eisenbud's method. 
It can be applied in the form presented in \ref{stabilization} since the maximal ideal $\m$ is generated by the
regular sequence $x_1,\dots,x_n$.
Writing $w = \sum x_i w_i$, we obtain the factorization $k^{\stab}$ as
\[
(\bigoplus_{i=0}^n {\textstyle\bigwedge\nolimits}^{i} V, \;s_0 + s_1 ) \text{,}
\]
where $s_0$ denotes contraction with $(x_1,\dots,x_n)$ and $s_1$ is given by exterior
multiplication by $(w_1, \dots, w_n)$.

In supergeometric terms, $k^{\stab}$ is given by the superalgebra 
$R\left<\theta_1, \ldots, \theta_n\right>$ with odd differential
operator $\delta = \sum \delta_i$ with
\[
\delta_i = x_i \deli + w_i \theta_i \text{.}
\]
We denote the $\Z/2$-graded $R$-module of all polynomial differential operators
on $R\left<\theta_1, \ldots, \theta_n\right>$ by $A$ and obtain the description
\[
\MF(k^{\stab}, k^{\stab}) \cong (A,\; [\delta, -]) \text{.}
\]

We will now reduce the proof of Theorem \ref{generator} to a statement in homological
algebra. 

\subsection{Maximal Cohen-Macaulay approximations via matrix factorizations}

Let $\Mod(S)$ be the stable
category of $S$-modules. The objects are arbitrary $S$-modules whereas the
morphisms $\underline{\Hom}_S(M,N)$ are defined to be the quotient of $\Hom_S(M,N)$ by the two-sided ideal of
morphisms factoring through some free $S$-module. Analogously to the case of
finitely generated modules, the category $\left[ \MF^\infty(R,w) \right]$ is equivalent
to a full subcategory of $\Mod(S)$ which we denote by
$\underline{\MCM}^{\infty}(S)$. For our purposes, it is not necessary to give a
characterization of the class of $S$-modules which form the objects of
$\underline{\MCM}^{\infty}(S)$. We may think of them as generalized maximal
Cohen-Macaulay modules (cf. \cite{chen.xw}).

We will often use the following lemma which corresponds to Buchweitz's notion of
maximal Cohen-Macaulay approximation (cf. \cite{buchweitz}). 

\begin{lem}\label{mcm.approx}
	Let $X$ and $Y$ be objects
	in $\MF^{\infty}(R,w)$ and assume that $X$ has finite rank. Let $A$ be the endomorphism dg
	algebra of $X$ and assume that $Y$ is the
	stabilization of the $S$-module $L$. Then there exists a natural
	isomorphism
	\[
	\MF(X, Y) \cong \Hom_R^{\Zt}(X, L)
	\]
	in the category $\D(A^{\op})$.
\end{lem}
\begin{proof}
	We abbreviate the functor $-
	\otimes_R S$ by $\overline{-}$.	
	Let $Z$ be the $\Z$-graded total product complex of the double complex
	\[
	\xymatrix{
	& \vdots \ar[d]  & \ar[d] \vdots & \ar[d] \vdots & \\
\dots \ar[r] & \ar[d] \Hom( \overline{X^1}, \overline{Y^0}) \ar[r] &	\Hom( \overline{X^0}, \overline{Y^0}) \ar[d] \ar[r] & \ar[d] \Hom( \overline{X^1}, \overline{Y^0}) \ar[d] \ar[r] &  \dots\\
\dots \ar[r] & \ar[d] \ar[r] \Hom( \overline{X^1}, \overline{Y^1}) &	\Hom( \overline{X^0}, \overline{Y^1}) \ar[d] \ar[r] & \ar[d] \ar[r] \Hom( \overline{X^1}, \overline{Y^1}) \ar[d] \ar[r]  & \dots\\
\dots \ar[r] &  \ar[r] \Hom( \overline{X^1}, \overline{Y^0})       &	\Hom(
\overline{X^0}, \overline{Y^0})\ar[r]         &  \ar[r] \Hom( \overline{X^1},
\overline{Y^0})   \ar[r] & \dots
	}
	\]
	Since the complex $Z$ is $2$-periodic, we may think of it as a
	$\Zt$-graded complex $Z^{\Zt}$.
	Consider the natural map 
	\[
	\MF( X, Y) \lra Z^{\Zt}
	\]
	which is given by reducing modulo $w$ and extending $2$-periodically. We claim that this map is
	a quasi-isomorphism. We prove the surjectivity on $\H^0$.
	Assume an element of $\H^0(Z)$ is represented by the map of complexes
	\[
	\xymatrix{
	\dots \ar[r]  & \overline{X^0} \ar[d]^{f^2} \ar[r] & \overline{X^1} \ar[d]^{f^1} \ar[r] &
	\overline{X^0} \ar[d]^{f^0} \\
	\dots \ar[r] & \overline{Y^0} \ar[r] & \overline{Y^1} \ar[r] &
	\overline{Y^0} \text{.}
	}
	\]
	The collection $\{f^i\}$ induces a map
	$f$ between the maximal Cohen-Macaulay modules $M$ and $N$ which
	are resolved by $\overline{X}$ resp. $\overline{Y}$.
	Alternatively, we can lift $f$ to a map of complexes 
	\[
	\xymatrix{
	X^1 \ar[r]^{\varphi} \ar[d]_{\widetilde{f^1}} & X^0
	\ar[d]^{\widetilde{f^0}}\\
	Y^1 \ar[r]^{\varphi} & Y^0 \text{.}
	}
	\]
	Using the relation $\varphi \circ \psi = w$ one immediately checks that
	the maps ${ \widetilde{f^i} }$ also commute with $\psi$. Reducing
	${ \widetilde{f^i} }$ modulo $w$ and extending periodically, we obtain a map
	of complexes representing $f$ which therefore has to be homotopic to
	$\{f^i\}$. A similar argument shows injectivity on $\H^0$.

	Next, denote by $W$ the $\Z$-graded total product complex of the double complex
	\[
	\xymatrix{
	& \vdots \ar[d]  & \ar[d] \vdots & \ar[d] \vdots & \\
\dots \ar[r] & \ar[d] \ar[r] \Hom( \overline{X^1}, \overline{Y^1}) &	\Hom( \overline{X^0}, \overline{Y^1}) \ar[d] \ar[r] & \ar[d] \ar[r] \Hom( \overline{X^1}, \overline{Y^1}) \ar[d] \ar[r]  & \dots\\
\dots \ar[r] & \ar[d] \Hom( \overline{X^1}, \overline{Y^0}) \ar[r] &	\Hom( \overline{X^0}, \overline{Y^0}) \ar[d] \ar[r] & \ar[d] \Hom( \overline{X^1}, \overline{Y^0}) \ar[r] \ar[d] \ar[r] &  \dots\\
\dots \ar[r] & \ar[d] \Hom( \overline{X^1}, \overline{Y^{-1}}) \ar[r] &	\Hom( \overline{X^0}, \overline{Y^{-1}}) \ar[d] \ar[r] & \ar[d] \Hom( \overline{X^1}, \overline{Y^{-1}}) \ar[d] \ar[r] &  \dots\\
	& \vdots \ar[d]  & \ar[d] \vdots & \ar[d] \vdots & 
	\\
	\dots \ar[r] &  \ar[r] \Hom( \overline{X^1}, \overline{Y^{-r}})       &	\Hom(
	\overline{X^0}, \overline{Y^{-r}})\ar[r]         &  \ar[r] \Hom( \overline{X^1},
	\overline{Y^{-r}})  & \dots
	}
	\]
	where the complex 
	\[
	\xymatrix{
		\dots \ar[r] & \overline{Y^0} \ar[r] & \overline{Y^1} \ar[r] & \overline{Y^0}
		\ar[r] & \overline{Y^{-1}} \ar[r] & \dots \ar[r] & \overline{Y^{-r}}
		}
	\]
	is a resolution of the $S$-module $L$. Since by assumption $N$ is an
	even syzygy of $L$, such a resolution exists and $r$ is an even number.
	Again, due to $2$-periodicity, we may think of $W$ as a
	$\Zt$-graded complex $W^{\Zt}$. There is an obvious projection map
	\[
	p : \; W \to Z
	\]
	which we claim to be a quasi-isomorphism.
	This
	amounts to showing that any map of complexes
	\[
	\xymatrix{
	\dots \ar[r]  & \overline{X^0} \ar[d]^{f^2} \ar[r] & \overline{X^1} \ar[d]^{f^1} \ar[r] &
	\overline{X^0} \ar[d]^{f^0} \\
	\dots \ar[r] & \overline{Y^0} \ar[r] & \overline{Y^1} \ar[r] & \overline{Y^0}
	}
	\]
	can be extended, uniquely up to homotopy, to a map of complexes
	\[
	\xymatrix{
	\dots \ar[r]  & \overline{X^0} \ar[d]^{f^2} \ar[r] & \overline{X^1} \ar[d]^{f^1} \ar[r] &
	\overline{X^0} \ar[d]^{f^0} \ar[r] & \overline{X^1} \ar[d]^{f^{-1}} \ar[r] &
	\dots \ar[r] & \overline{X^0} \ar[d]^{f^{-r}}\\
	\dots \ar[r] & \overline{Y^0} \ar[r] & \overline{Y^1} \ar[r] &
	\overline{Y^0} \ar[r] & \overline{Y^{-1}} \ar[r] & \dots \ar[r] &
	\overline{Y^{-r}} \text{.}
	}
	\]
	Using the short exact sequence 
	\[
	0  \lra M \lra  \overline{X^1} \lra \syz(M) \lra 0
	\]
	one shows that the obstruction for the existence of the map $f^{-1}$
	lies in the space $\Ext^1(\syz(M), \overline{Y^{-1}})$ which vanishes
	since $\syz(M)$ is maximal Cohen-Macaulay. Here we use that $M$ (and 
	thus $\syz(M)$) is a finitely generated $S$-modules which implies that 
	$\Ext^i(\syz(M),-)$ commutes with coproducts. Further, one
	obtains that any two lifts differ by the pullback of a map in
	$\Hom_S(\syz(M),\overline{Y^{-1}})$. The obstruction for lifting this
	difference to a map in $\Hom_S(\overline{X^0},\overline{Y^{-1}})$ (which will define the
	desired homotopy) lies in $\Ext^1(M, \overline{Y^{-1}}) = 0$.
	An iteration of this argument leads to the desired extension which is
	unique up to homotopy.

	Finally, the complex $W^{\Zt}$ admits an augmentation map to the complex
	$\Hom^{\Zt}_R(X, L)$ which is a quasi-isomorphism by
	\cite[5.5.11]{weibel}.

	We conclude by noting that we have constructed quasi-isomorphisms of the form
	\[
	\MF(X, Y) \stackrel{\simeq}{\lra}  Z^{\Zt}
	\stackrel{\simeq}{\longleftarrow}
	W^{\Zt} \stackrel{\simeq}{\lra} \Hom^{\Zt}_R(X, L)\text{,}
	\]
	inducing the claimed isomorphism in the category $\D(A^{\op})$.
\end{proof}

The next proposition gives an explicit simple description of the map
\[
	\MF(X, Y) \cong \Hom_R^{\Zt}(X, L)
\]
in the case where $L$ is a complete intersection module.

\begin{prop} \label{exp.approx} Let $L$ be a complete intersection module as in
	\ref{stabilization}. Let $Y = (\bigoplus_{i=0}^m {\textstyle\bigwedge\nolimits}^{i} V, \;s_0 + s_1 )$ be the stabilization of $L$ constructed in
	Corollary \ref{cor.stab}. Define the map $g$ as the composition
	\[
	\bigoplus_{i=0}^m {\textstyle\bigwedge\nolimits}^{i} V \to
	{\textstyle\bigwedge\nolimits}^{0} V \to L \text{,}
	\] 
	where the first map is the projection and the second map is the natural
	quotient map realizing $(\bigoplus_{i=0}^m
	{\textstyle\bigwedge\nolimits}^{i} V, \;s_0)$ as an $R$-free resolution
	of $L$. Then post-composition by $g$ induces a quasi-isomorphism
	\[
	\MF(X, Y) \to \Hom_R^{\Zt}(X, L) \text{.}
	\]
\end{prop}

\begin{proof}
	Inspecting the explicit form of the $S$-free resolution of $L$ constructed in Lemma
	\ref{sfree}, one directly checks that a possible choice of the
	extension $f^{-r}$ constructed in the proof of Lemma \ref{mcm.approx} is
	obtained by post-composition with $g$.
\end{proof}

Note that in light of the general theory of MCM approximations developed in
\cite{buchweitz} the map $g$ is nothing else than an explicit description (on
the level of matrix factorizations) of the approximation map from the maximal
Cohen-Macaulay module corresponding to $Y$ to $L$.

\subsection{Reduction to the Homological Nakayama Lemma}

Given a factorization $X$ in $\MF(R,w)$ we can define a dual factorization
$X^\vee = \Hom_R^{\Zt}(X,R)$. Note however that, due to the usual Koszul signs rule, the
factorization $X^\vee$ is naturally an object of the category $\MF(R,-w)$.

\begin{prop} \label{residue.dual}
Let $Y$ be an object in $\MF^{\infty}(R,w)$ with corresponding $S$-module $N$ in
$\underline{\MCM}^\infty(S)$.
For all $i>0$, we have a natural isomorphism 
\[
[k^{\stab}[n+i],Y] \cong
\Tor_{i}^S(k, N) \text{,}
\]
where $n$ is the (parity of) the dimension of $R$. 
\end{prop}

\begin{proof}
	We apply Lemma \ref{mcm.approx} to obtain a quasi-isomorphism
	\[
	\MF(k^{\stab}[n+i],Y) \to \Hom^{\Zt}_R(k^{\stab}[n+i],N) \text{.}
	\]
	The right hand complex can be rewritten as
	\begin{equation}\label{tor.com}
	\overline{(k^{\stab}[n+i])^{\vee}} \otimes_S N \text{.}
	\end{equation}
	A direct inspection of the stabilization construction, using
	the fact that a Koszul complex is self-dual, yields that the
	factorization
	$(k^{\stab}[n])^{\vee}$ 
	in the category $\MF(R,-w)$ is a stabilization of the residue field $k$. Thus, if $M$ denotes the stable even
	syzygy of $k$, then the $0$-th cohomology of the complex \ref{tor.com}
	computes $\Tor^S_i(M,N)$. Next, observe that the modules $\Tor^S_i(M,N)$
	and $\Tor^S_i(k,N)$ are isomorphic for $i>>0$. But this means that they
	must be isomorphic for all positive $i$ since $N$ admits a $2$-periodic
	free resolution (and thus both $\Tor$-modules are $2$-periodic for
	positive $i$).
\end{proof}

Note that, since $\Tor_i^S(k,N)$ is $2$-periodic, vanishing for $i \in \{1,2\}$
implies vanishing for all $i > 0$.
Therefore, we reduced Theorem \ref{generator} to a statement which we might call
the \emph{Homological Nakayama Lemma}: a
module $N$ in $\underline{\MCM}^{\infty}(S)$ is free if and only if
$\Tor_i^S(k,N) = 0$ for all $i > 0$. For general infinitely generated
$S$-modules this statement is certainly false, but it turns out to be true for
modules in $\underline{\MCM}^{\infty}(S)$ if we assume that $(R,w)$ has isolated
singular locus. We will give the proof in the next subsection.

\subsection{The Homological Nakayama Lemma} \label{subsect.nakayama}

Consider a matrix factorization
\[
\xymatrix{ 
X^1 \ar@<+.5ex>[r]^{\varphi} & \ar@<+.5ex>[l]^{\psi} X^0
}
\]
where $X^0$ and $X^1$ are free $R$-modules of possibly infinite rank and let 
$M = \coker(\varphi)$ be the corresponding $S$-module. As already explained, applying $- \otimes_R
S$ to $X$ and extending periodically one obtains an $S$-free resolution 
\[
\xymatrix@C=0.7cm{ 
 \dots \ar[r] & \overline{X^0} \ar[r]^{d} &  \overline{X^1} \ar[r]^{d} &
\ar[r] \overline{X^0} & M \ar[r] & 0
}
\]
of $M$.
The formula $\partial_k w = \partial_k (\varphi \psi) = \partial_k (\varphi)
\psi + \varphi \partial_k (\psi)$ establishes that multiplication by $\partial_k
w$ is homotopic to zero on the endomorphism complex of $X$. Interpreting this fact in
terms of the resolution $\overline{X}$ of $M$ one easily deduces the following
fundamental observation.

\begin{lem} \label{lem} For any $S$-module $N$, multiplication by $\partial_k w$ annihilates the $S$-module $\Tor^S_i(N,M)$ for all $i>0$.
\end{lem}

We will use the following result from \cite{jensen}. 

\begin{theo}[Gruson, Jensen] \label{pd} The projective dimension of an arbitrary flat $S$-module is at most
$n-1$.
\end{theo}

Recall that the \emph{Tyurina algebra} is defined to be $\Omega_w = S/(\partial_1 w , \dots, \partial_n w)$.

\begin{theo}
\label{nakayama}
Let $X$ be a matrix factorization of possibly infinite rank and let $M =
\coker(\varphi)$. Then the following are equivalent:
\begin{enumerate}[(1)]
\item $M$ is a free $S$-module.
\item $M$ is a flat $S$-module.
\item $\Tor^S_i(N,M) = 0$ for every finitely generated $S$-module $N$ and $i>0$.
\item $\Tor^S_i(N,M) = 0$ for every finitely generated $\Omega_w$-module $N$ and $i>0$.
\end{enumerate}
If $w$ has an isolated singularity then the above are equivalent to 
\begin{enumerate}[(1)]
\item[(5)] $\Tor^S_i(k,M) = 0$ for $i>0$.
\end{enumerate}
\end{theo}

\begin{proof}
The implications (1) $\Rightarrow$ (2) $\Rightarrow$ (3) $\Rightarrow$ (4)
$\Rightarrow$ (5) are obvious.\\

\noindent
(2) $\Rightarrow$ (1): The module $M$ is flat and has therefore finite
projective dimension by Theorem \ref{pd}. This implies that $\syz^k(M)$ is
projective for $k >> 0$. Since $M$ has a $2$-periodic resolution, we have $M
\cong \syz^k(M)$ for every even natural number $k$. 
So $M$ is projective and since $S$ is local Kaplansky's theorem implies that $M$ is free.\\

\noindent
(3) $\Rightarrow$ (2): This is the standard homological criterion for
flatness.\\

\noindent
(4) $\Rightarrow$ (3): Let us fix a non-zero partial derivative $\partial w =
\partial_k w$ and let $N$ be a finitely generated $S$-module. Since $R$ a
regular local ring, it is in particular an integral domain and thus
multiplication by $\partial w$ is injective on $R$.
We compute the cohomology of the complex 
\[
R/\partial w \otimes^{\L}_R M \otimes^{\L}_S N
\]
in two different ways. Namely, this complex is
quasi-isomorphic to the total complex of the double complex
\[
\xymatrix{
\vdots \ar[d]  & \ar[d] \vdots & \ar[d] \vdots & \\
\overline{X^0} \otimes_S N \ar[d]^{d} & \ar[d]^{d} \ar[l]_{\partial w}
\overline{X^0} \otimes_S N & 0 \ar[d] \ar[l] & \ar[l] \dots\\
\overline{X^1} \otimes_S N \ar[d]^{d} & \ar[d]^{d} \ar[l]_{-\partial w}
\overline{X^1} \otimes_S N & 0 \ar[d] \ar[l] & \ar[l] \dots\\
\overline{X^0} \otimes_S N  &  \ar[l]_{\partial w}
\overline{X^0} \otimes_S N & 0 \ar[l] & \ar[l] \dots
}
\]
which we may filter horizontally as well as vertically. Both filtrations lead to
spectral sequences converging strongly to the target $\H^*(R/\partial w \otimes^{\L}_R M
\otimes^{\L}_S N)$.
The vertical filtration leads to a spectral sequence with $E^1$ given by 
\[
\xymatrix{
\vdots  & \vdots & \vdots \\
\Tor^S_2(M,N) & \ar[l]_{\partial w}
\Tor^S_2(M,N) & 0 \ar[l] & \dots \\
\Tor^S_1(M,N)  &  \ar[l]_{-\partial w} \Tor^S_1(M,N) & 0 \ar[l] & \dots \\
M \otimes_S N  &  \ar[l]_{\partial w} M \otimes_S N & 0 \ar[l] & \dots }
\]
It degenerates at $E^2$ which, using Lemma \ref{lem}, is given by
\[
\xymatrix{
\vdots  & \vdots & \vdots & \\
\Tor^S_2(M,N) & 
\Tor^S_2(M,N) & 0 & \dots \\
\Tor^S_1(M,N)  &  \Tor^S_1(M,N) & 0 & \dots  \\
M \otimes_S N/\partial w  &  \Tor^S_1( M \otimes_S N, S/\partial w) & 0 & \dots
}
\]
This implies that $\Tor^S_i(M,N) = 0$ for $i > 0$ if and only if $\H^j(R/\partial
w \otimes^{\L}_R M \otimes^{\L}_S N) = 0$ for $j \ge 2$.\\
Using the horizontal filtration of the above double complex we obtain a spectral
sequence with first page 
\[
\xymatrix{
\vdots \ar[d]  & \ar[d] \vdots & \ar[d] \vdots \\
\overline{X^0} \otimes_S N/\partial w \ar[d]^{d} & \ar[d]^{d} \overline{X^0}
\otimes_S \Tor^R_1( N, R/\partial w) & 0 \ar[d] & \dots \\
\overline{X^1} \otimes_S N/\partial w \ar[d]^{d} & \ar[d]^{d} \overline{X^1}
\otimes_S \Tor^R_1( N, R/\partial w) & 0 \ar[d] & \dots \\
\overline{X^0} \otimes_S N/\partial w  & \overline{X^0} \otimes_S \Tor^R_1( N,
R/\partial w) & 0 & \dots \\
}
\]
and $E^2$ given by
\[
\xymatrix{
\vdots & \vdots & \vdots \\
\Tor^S_2(M, N/\partial w) \ar[ddr]  & \Tor^S_2(M, \Tor^R_1( N, R/\partial w)) &
0  & \dots \\
\Tor^S_1(M, N/\partial w)  & \Tor^S_1(M, \Tor^R_1( N, R/\partial w)) & 0 & \dots  \\
M \otimes_S N/\partial w  & M \otimes_S \Tor^R_1( N, R/\partial w) & 0 & \dots \\
}
\]
Observe that both $S$-modules $N/\partial w$ and $\Tor^R_1( N, R/\partial w)$
are finitely generated and annihilated by $\partial w$. This allows us to
conclude that $\H^j(R/\partial w \otimes^{\L}_R M \otimes^{\L}_S N) = 0$ for $j \ge 2$ if $\Tor^S_i(M, -)$ vanishes
on all finitely generated $S/\partial w$-modules for $i > 0$.

Applying this construction iteratively to all non-zero partial derivatives of $w$ yields the implication.\\

\noindent
(5) $\Rightarrow$ (4):
These are standard arguments for modules over Artinian rings. Let $N$ be a
finitely generated $\Omega_w$-module. It admits a finite filtration with successive quotients
isomorphic to cyclic $\Omega_w$-modules. If $\Omega_w/I$ is such a cyclic
module then we obtain a short exact sequence
\[
\xymatrix{
0 \ar[r] & K \ar[r] & \Omega_w/I \ar[r] & \Omega_w/\m \ar[r] & 0 
}
\]
where by assumption $\Omega_w/\m \cong k$. The claim follows inductively from inspection of the
associated long exact $\Tor$-sequence. 
\end{proof}

\subsection{A counterexample}
An example of a non-isolated singularity for which the Homological Nakayama Lemma does not
hold can be constructed as follows. We consider the local $k$-algebra $R = k[ [x,y]]$
with $w=xy^2$. The hypersurface algebra is given by $S = R/xy^2$
and we consider the $S$-module of formal Laurent series $k( ( x))$. Consider the matrix factorization with underlying $\Zt$-graded $R$-module
\begin{align*}
X^1 &= \bigoplus_{i \in \Z} R h_i \oplus \bigoplus_{i \in \Z} R e_i\\
X^0 &= \bigoplus_{i \in \Z} R f_i \oplus \bigoplus_{i \in \Z} R g_i 
\end{align*}
with $\varphi$ given by the assigments
\begin{align*}
h_i & \mapsto y f_i - x g_i + g_{i+1}\\
e_i & \mapsto xy g_i
\end{align*}
and $\psi$ defined by
\begin{align*}
f_i & \mapsto xy h_i + x e_i - e_{i+1}\\
g_i & \mapsto y e_i \text{.}
\end{align*}
The cokernel $M$ of $\varphi$ is a first syzygy module of the
$S$-module $k( ( x))$. In fact, the factorization was found by imitating
Eisenbud's method which in general only works for finitely generated modules.
Using the factorization and the resulting $2$-periodic $S$-free resolution of
$M$ one checks that $\Tor_i^S(k,M) = 0$ for all $i > 0$. But $M$ is not free and
in fact, consistently with Theorem \ref{nakayama}, we have $\Tor_i^S(\Omega_w, M) \ne
0$ for both $i = 1$ and $i = 2$.

\subsection{Localization}
\label{subsect.localization}

In this section we generalize the generation results of \ref{subsect.nakayama}
to nonlocal ambient rings. 
Let $(R,w)$ be as in \ref{subsect.nonlocal}, satisfying condition (\ref{a}).
We will first show that the category $\MF^\infty(R,w)$ has a compact generator
and then deduce that the localization functor
\[
\MF^\infty(R,w) \to \MF^\infty(R_{\s}, w)
\]
is a quasi-equivalence. 
We point out that a localization statement of this type was already proved
in \cite[1.14]{orlov-2003}. 
As a first step, we have the 
following variant of Proposition \ref{residue.dual}, again reducing the compact generation problem to the
Homological Nakayama Lemma.

\begin{prop} \label{residue.dual.2}
We consider the residue field $k$ of $\s$ as an
$S$-module and consider its stabilization $k^{\stab}_{(R,-w)}$ in the category
$\MF(R,-w)$. For an object $Y$ in $\MF^\infty(R,w)$ with corresponding MCM
module $N$ in $\underline{\MCM}^{\infty}(S)$ and $i>0$, 
we have a natural isomorphism
\[
[(k^{\stab}_{(R,-w)})^{\vee}[i],Y] \cong \Tor_{i}^S(k, N)\text{.}
\]
\end{prop}

Note, that we may not be able to
obtain an explicit description of $(k^{\stab}_{(R,-w)})^{\vee}$ by using stabilization method of \ref{stabilization}, since the
ideal $\s$ is not necessarily generated globally by a regular sequence. This is
not a problem, all which matters
is that $(k^{\stab}_{(R,-w)})^{\vee}$ is a compact object in $\MF^\infty(R,w)$.

Next we claim that the Homological Nakayama Lemma generalizes to
the nonlocal hypersurface ring $S=R/w$. In fact, we can immediately reduce to the local situation
via the following lemma.

\begin{lem} \label{tor.loc}
	Let $M,N$ be objects in $\underline{\MCM}^{\infty}(S)$. Then the
	$S$-module $\Tor_i^S(N,M)$ is supported on $\s$ for $i>0$. In particular,
	we have an isomorphism $\Tor_i^S(N,M) \cong \Tor_i^{S_{\s}}(M_\s,N_\s)$.
\end{lem}
\begin{proof}
	The $S$-module $\Tor_i^S(M,N)$ for positive $i$ is calculated by the cohomology of the
	$2$-periodic complex $\overline{X} \otimes_S N$, where $X$ is the matrix
	factorization corresponding to $M$. After localizing this complex at any
	prime ideal $\p$, we can express $dw = \partial_1(w) e_1 + \dots +
	\partial_n(w) e_n$ where $e_1, \dots, e_n$ is a basis of
	$\Omega_{R_\p/k}$ with dual derivations $\partial_1, \dots, \partial_n$.
	The same argumentation as in the beginning of Subsection
	\ref{subsect.nakayama} shows that multiplication by $\partial_i(w)$ on
	the complex $S_\p \otimes_S \overline{X} \otimes_S N$ is homotopic to
	$0$. For $\p \ne \s$ at least one partial derivative is invertible which implies that
	the localized complex is contractible. The statement of the lemma 
	follows since localization commutes with taking cohomology.
\end{proof}

\begin{cor} The object $(k^{\stab}_{(R,-w)})^{\vee}$ is a compact generator of
	the category $\MF^{\infty}(R,w)$.
\end{cor}

\begin{proof}
	This follows by combining Proposition \ref{residue.dual.2}, Lemma
	\ref{tor.loc} and Theorem \ref{nakayama}.
\end{proof}

\begin{theo}
	\label{equiv.localization}
The localization functor
\[
\MF^\infty(R,w) \to \MF^\infty(R_{\s}, w)
\]
is a quasi-equivalence.
\end{theo}

\begin{proof}
	By the same argument as in the proof of Theorem \ref{genequivalence}, 
	it suffices to check the quasi-fully faithfullness on the compact
	generator $X = (k^{\stab}_{(R,-w)})^{\vee}$. Let $N$ denote the maximal
	Cohen-Macaulay module corresponding to $X$. The localization map 
	\[
	\MF( X,X) \to \MF(R_{\s} \otimes_R X, R_{\s} \otimes_R X)
	\]
	induces by Proposition \ref{residue.dual.2} on cohomology the localization map
	\[
	\Tor_{i}^S(k, N) \to \Tor_{i}^{S_\s}(k, N_\s)
	\]
	which is an isomorphism by Lemma \ref{tor.loc}.

	The stabilization of the residue field at $\s$ in $\MF^\infty(R,w)$ maps to an object in $\MF^\infty(R_{\s},w)$ which is
	isomorphic to $k^{\stab}$, and therefore generates the category $\MF^\infty(R_{\s},w)$. 
	Thus, the localization functor is essentially surjective on the homotopy
	categories. 
\end{proof}

\begin{cor} \label{global.generator} Using the above notation, let $k$ denote the residue field of the
	maximal ideal $\s \in S$ and let $k^{\stab}$ be its stabilization in the
	category $\MF(R,w)$. Then $k^{\stab}$ is a compact generator of
	$\MF^{\infty}(R,w)$.
\end{cor}

\begin{proof} This follows since the localization functor commutes with
	coproducts and maps $k^{\stab}$ to a
	stabilization of the residue field in the category
	$\MF^{\infty}(R_\s,w)$. Since the latter object is a compact generator,
	$k^{\stab}$ generates as well.
\end{proof}

\section{First applications}
\label{sect.app}

\subsection{The homotopy theory of $2$-periodic dg categories}
\label{subsect.homotopy}

Before giving applications of Theorem \ref{generator}, we introduce a
homotopical framework for $2$-periodic dg categories. Most statements in this
subsection are immediate consequences or variants of well-known results. We define
model structures in the $2$-periodic context which allow us to obtain
a homotopy theory analogous to the one developed in \cite{tabuada, toen.morita}. All dg categories are assumed to be
small by virtue of choosing small quasi-equivalent dg categories. We refer to 
loc. cit. for details on how to take the necessary set-theoretic
precautions. Our model category terminology is the one used in \cite{hovey} which also
contains the standard results we need.

For a field $k$ consider the dg algebra $\ku$ where the variable $u$ has
degree $2$ and the differential is the zero map. Let $\C(k)$ be the dg category
of unbounded complexes of $k$-modules and define $\C(\ku)$ to be the dg category
of functors from $\ku$, considered as a dg category with one object, to $\C(k)$.
There is an obvious enriched equivalence between the dg category of $\Zt$-graded
complexes over $k$ and the category $\C(\ku)$. 

There is an adjunction
\[
\xymatrix@C=2cm{
\C(k) \ar[r]^(.35){-\otimes_k \ku} & \C(\ku) \quad \text{,} \quad \C(k)
& \ar[l]_(.4){F} \C(\ku) 
}
\]
where $F$ denotes the forgetful functor. The image under $-\otimes_k \ku$ of the
generating (trivial) cofibrations for the projective model structure on $\C(k)$,
as defined in \cite{hovey},
form the generating (trivial) cofibrations for a model structure on $\C(\ku)$.
Therefore, $\C(\ku)$ admits a cofibrantly generated model structure such that the
above adjunction is a Quillen adjunction. As for the category $\C(k)$, weak equivalences in
$\C(\ku)$ are defined to be quasi-isomorphisms and fibrations are levelwise surjective
maps. Since $k$ is a field, every object in $\C(k)$ and $\C(\ku)$ is cofibrant.
Note that $\C(\ku)$ has a monoidal structure given by the tensor product over
$\ku$. Under the equivalence with $\Zt$-graded complexes, this tensor product
translates into the $\Zt$-graded tensor product.

Let $\dgk$ denote the category of small dg categories over $k$. We introduce the
category $\dgku$ of $2$-periodic dg categories where the objects are small
categories enriched over $\C(\ku)$ and the morphisms are dg functors. The above
adjunction on the level of complexes induces an adjunction
\[
\xymatrix@C=2cm{
\dgk \ar[r]^(0.35){-\otimes_k \ku} & \dgku \quad \text{,} \quad \dgk &
\ar[l]_(0.4){F} \dgku \text{.}
}
\]
We claim that $\dgku$ admits the structure of a cofibrantly generated model
category such that this adjunction is a Quillen adjunction. The category $\dgk$
admits a cofibrantly generated model structure which is described in
\cite{tabuada}. Again, we can take the generating (trivial) cofibrations in
$\dgku$ to be the image under $-\otimes \ku$ of the generating (trivial)
cofibrations for $\dgk$. The only property needed to apply \cite[Theorem 2.1.19]{hovey} which does
not formally follow from the above adjunction is that relative $J$-cell
complexes are weak equivalences. However, the proof of this fact for $\dgk$
given in \cite{tabuada} implies the statement for $\dgku$ in complete analogy.
As for $\dgk$, the weak equivalences in $\dgku$ are quasi-equivalences.

The existence of this model structure allows for a description of the mapping
spaces in the category $\Ho(\dgku)$ as done for $\Ho(\dgk)$ in
\cite{toen.morita}.

Note that $\dgku$ admits a closed monoidal structure where the tensor product
of two dg categories is given by the product on objects and the tensor product
over $\ku$ on morphism complexes. Thus, we have an adjunction
\[
\Hom(T \otimes T', T'') = \Hom(T, \underline{\Hom}(T', T''))
\] 
where $\underline{\Hom}$ denotes the dg category of $\C(\ku)$-enriched functors.
For a category $T$ in $\dgku$, we define the dg category of modules over $T$ to
be
\[
\mmod{T} = \underline{\Hom}(T, \C(\ku)) \text{.}
\]
For an object $x$ in $T$ we define $\underline{h}_x$ to be the $T^{\op}$-module
given by $\underline{h}_x(y) = T(y,x)$. The dg Yoneda functor
\[
\underline{h}_{-}:\; T \lra \mmod{T^{\op}}, \; x \mapsto
\underline{h}_x
\]
is $\C(\ku)$-fully faithful (i.e. induces an isomorphism of morphism complexes).
Dually we have a $\C(\ku)$-fully faithful functor 
\[
\underline{h}^{-}:\; T^{\op} \lra \mmod{T}, \; x \mapsto
T(x,-) \text{.}
\]
Functors of the form $\underline{h}_x$ are called
\emph{representable}, the ones of the form $\underline{h}^x$
\emph{corepresentable}.

There exists a $\C(\ku)$-model structure on $\mmod{T}$ where the fibrations and weak
equivalences are defined levelwise using the model structure on $\C(\ku)$. This
model structure is cofibrantly generated where generating (trivial) cofibrations are obtained by
applying the functors $\underline{h}^x \otimes -$ to generating (trivial)
cofibrations in $\C(\ku)$ for all $x \in T$.
The homotopy category $\Ho(\mmod{T})$ yields the \emph{derived category of} $T$ which we
denote by $\D(T)$. Note that due to the existence of the model structure we can
define $\Int(\mmod{T})$ to be the full dg subcategory of $T$ consisting of objects
which are both fibrant and cofibrant.
As in the last subsection we denote by $[T]$ the category obtained from a dg
category $T$ by applying the functor $\H^0(-)$ to all morphism complexes. This
yields a functor $[-]$ from $\dgku$ to $\operatorname{cat}$.
We have a natural equivalence of categories $[ \Int(\mmod{T}) ] \simeq \D(T)$.

Since the representable $T^{\op}$-modules $\underline{h}_x$ are cofibrant
and fibrant, the Yoneda embedding yields a functor
\[
T \lra \Int( \mmod{T^{\op}}) \text{.}
\]
To simplify notation we introduce $\widehat{T} = \Int( \mmod{T^{\op}})$.

A $T^{\op}$-module $M$ is called \emph{compact} or \emph{perfect} if $[M, -]$ commutes with coproducts.
It is easy to see that all representable modules are perfect. Therefore,
the Yoneda embedding provides a functor
\[
T \lra \widehat{T}_{\pe} \text{,}
\]
where the subscript indicates the full dg subcategory of perfect modules.
The category $\widehat{T}_{\pe}$ is called the \emph{triangulated hull} of $T$ and the
dg category $T$ is called \emph{triangulated} if the Yoneda embedding into its
triangulated hull is a quasi-equivalence.

We conclude by observing that the homotopical framework which we defined for
$2$-periodic categories is in exact analogy to the one defined in
\cite{toen.morita}. All properties pointed out in sections 2 and 3
of loc. cit. hold mutatis mutandis in the $2$-periodic context. Therefore all
proofs in sections 4,5,6 and 7 can be repeated more or less verbatim to obtain
identical results over $\ku$. In fact, as the author of loc. cit. points out in
the introduction, he expects the results to generalize to $M$-enriched
categories for certain very general monoidal model categories $M$. The $2$-periodic
version would then correspond to the choice $M = \C(\ku)$. We will therefore cite results
in loc. cit. without further comments, if the $2$-periodic reformulation is obvious.

\subsection{Equivalences of categories}

Let $T$ be a $2$-periodic dg category and consider a set $W$ of morphisms in $T$.
By \cite[8.7]{toen.morita} there exists a dg category $\L_W(T)$ and a morphism $l:
T \to \L_W(T)$ in $\Ho(\dgku)$ which is called the localization of $T$ with
respect to $W$. It enjoys the following universal property which determines
it uniquely up to isomorphism in $\Ho(\dgku)$. For
every dg category $T'$ the pullback map
\[
l^*: \; [ \L_W(T), T' ] \lra [ T, T' ]
\]
is injective and the image consists of morphisms $f: T \to T'$ such that $[f]$
maps morphisms in $W$ to isomorphisms in $[T']$.  

For a $2$-periodic dg category $T$ we introduce the \emph{dg derived category
of}
$T$ as the localization $\L_W(\mmod{T})$ with respect to the set of weak
equivalences. The category $[ \L_W(\mmod{T})]$ is equivalent to the 
derived category $\D(T)$ of $T$ which we introduced in Section
\ref{subsect.homotopy}.

The essential arguments in the proof of the following theorem are due to Keller
\cite[4.3]{keller}.

\begin{theo} 
	\label{genequivalence}
	
	Let $T$ be a triangulated $2$-periodic dg category which admits coproducts. Let $S$ be a full dg
	subcategory of $T$ whose objects are compact in $[T]$. Assume that the
	smallest triangulated subcategory of $[T]$ which contains the objects of
	$[S]$ and is closed under coproducts is $[T]$ itself. Then the
	map 
	\[
	f:\; T \to \L_W(\mmod{S^{\op}}), \; x \mapsto l(T(-,x)|_S)
	\]
	is an isomorphism in $\Ho(\dgku)$. Furthermore, $f$ induces an
	isomorphism
	\[
	T_{\pe} \simeq \L_W(\mmod{S^{\op}})_{\pe}
	\]
	between the full dg subcategories of compact objects.
\end{theo}
\begin{proof}
	Note that, since both $T$ and $\L_W(\mmod{S^{\op}})$ are triangulated dg
	categories, the induced functor $[f] : [T] \to \D(S^{\op})$ is an exact
	functor of triangulated categories.

	We claim that $[f]$ commutes with coproducts. Indeed, if $\{x_i\}$ are
	objects in $T$ then the natural map 
	\[
	\coprod T(-, x_i)|_S \to T( -, \coprod x_i)|_S
	\]
	is a weak equivalence since the objects in $S$ are compact in $[T]$. It
	therefore becomes an isomorphism in $\D(S^{\op})$ which proves the
	claim.

	The restriction of $f$ to the dg subcategory $S$ factors over
	the weak equivalence $\Int(\mmod{S^{\op}}) \to \L_W(\mmod{S^{\op}})$ since representable modules
	are cofibrant in $\mmod{S^{\op}}$. Therefore, by the dg Yoneda lemma the
	restriction of $f$ to $S$ is quasi-fully faithful. 

	Consider the full subcategory $A$ of $[T]$ consisting of objects $x$ such that the map
	\[
	[T](s,x) \to \D(S^{\op})(f(s), f(x))
	\]
	is an isomorphism for all objects $s$ of $S$. By the five-lemma the category $A$ is triangulated and since $[f]$
	commutes with coproducts, $A$ contains coproducts. However, since we
	just saw that $A$ contains the objects of $S$, we have $A = [T]$ by
	assumption.
	Fixing an object $y$ in $T$ and applying the same argument to the
	subcategory formed by objects $x$ such that the map
	\[
	[T](x,y) \to \D(S^{\op})(f(x), f(y))
	\]
	is an isomorphism, we deduce that $[f]$ is fully faithful. This clearly
	implies that $f$ is quasi-fully faithful since $f$ is a map of
	triangulated dg categories.

	It remains to show that $[f]$ is essentially surjective. Since $[f]$
	commutes with coproducts, the essential image of $[f]$ contains the
	quasi-representable functors and is closed under coproducts. Using the Bousfield
	localization argument \cite[Lemma 2.2.1]{schwede} which we already used
	in Section \ref{sect.generator}, we conclude that the essential image must be all 
	of $\D(S^{\op})$.

	Since $[f]$ is an equivalence, its right adjoint is also a left
	adjoint and hence preserves coproducts. Thus $[f]$ and its adjoint
	preserve compactness implying the statement about the subcategories of
	compact objects.
\end{proof}

Recall the notation 
$\widehat{T} = \Int(\mmod{T^{\op}})$ for a $2$-periodic dg category $T$. The
natural map $\widehat{T} \to \L_W(\mmod{T^{\op}})$ is an isomorphism in $\Ho(\dgku)$. Hence, we may use
the dg category $\widehat{T}$ as an explicit model for the dg derived category of $T^{\op}$.

Recall the description of the dg algebra of endomorphisms of $k^{\stab}$
as the $\Zt$-graded algebra
\[
A = R\left<\theta_1, \ldots, \theta_n, \frac{\partial}{\partial \theta_1}, \dots, \frac{\partial}{\partial \theta_n} \right>
\]
of polynomial differential operators on $R\left<\theta_1, \ldots, \theta_n\right>$ equipped
with the differential $[\delta, -]$ where
\[
\delta = \sum_{i = 1}^{n} x_i \deli + w_i \theta_i \text{.}
\] 
We reserve the letter $A$ for this dg algebra throughout this subsection. We
slightly abuse notation and also use the symbol $A$ to refer to the corresponding $2$-periodic
dg category with a single object.

\begin{theo}
\label{equivalence}
Let $(R,w)$ be as in \ref{subsect.local} with isolated singular locus (\ref{is.sing}). Then there exist the following isomorphisms in $\Ho(\dgku)$.
\begin{enumerate}
	\item $\MF^{\infty}(R,w) \stackrel{\simeq}{\lra} \widehat{\MF(R,w)}$
	\item $\MF^{\infty}(R,w) \stackrel{\simeq}{\lra} \widehat{A}$
	\item $\widehat{\MF(R,w)}_{\pe} \stackrel{\simeq}{\lra} \widehat{A}_{\pe}$
\end{enumerate}
\end{theo}
\begin{proof}
	By Theorem \ref{generator} the object $k^{\stab}$ in $\MF^{\infty}(R,w)$
	is a compact generator and the corresponding full dg subcategory is
	isomorphic to $A$. In particular, the objects in the full dg subcategory $\MF(R,w)$
	generate $\MF^{\infty}(R,w)$. Since $\MF(R,w)$ also consists of compact
	objects we deduce the first two isomorphisms from Theorem
	\ref{genequivalence}. The last statement follows immediately from the
	second part of Theorem \ref{genequivalence}.
\end{proof}

We comment on the relevance of the previous theorem. The first isomorphism gives
a natural interpretation of the dg category $\MF^{\infty}(R,w)$ which we defined
in a somewhat ad hoc way. Namely, it is an explicit model for the 
dg derived category of $\MF(R,w)$. The second isomorphism will
turn out to be useful since we have an explicit description of the dg algebra
$A$. Finally, the last isomorphism identifies the triangulated hull of the dg
category $\MF(R,w)$. Note that the natural map $\MF(R,w) \to \widehat{\MF(R,w)}_{\pe}$ is
not necessarily an isomorphism in $\Ho(\dgku)$ since the triangulated category
$[\MF(R,w)]$ is not necessarily idempotent complete. We will determine an
explicit model for $\widehat{\MF(R,w)}_{\pe}$ in Section \ref{formal.completion}.

\begin{cor}\label{split}
	The triangulated category
$[\widehat{\MF(R,w)}_{\pe}]$ is equivalent to the smallest triangulated
subcategory of $[ \MF^{\infty}(R,w) ]$ which contains 
$k^{\stab}$ and is closed under summands.
\end{cor}

The next corollary introduces a notion of quasi-isomorphism between matrix
factorizations. The category $\left[ \MF^{\infty}(R,w) \right]$ is obtained as a
localization from $\MF^{\infty}(R,w)$ by inverting all such quasi-isomorphisms.
One may therefore think of the category $\left[ \MF^{\infty}(R,w) \right]$ as
a derived category of twisted complexes of $R$-modules.

\begin{cor}
	\label{cohomology}
Let $X, Y$ be objects in $\MF^{\infty}(R,w)$. A $0$-cycle $f \in \Hom(X,Y)$ induces an isomorphism in 
$\left[ \MF^{\infty}(R,w) \right]$ if and only if $f$ induces a
quasi-isomorphism of the complexes of $k$-modules $k \otimes_R X$ and $k
\otimes_R Y$.
\end{cor}

\begin{proof}
By Theorem \ref{equivalence}, $f$ induces an isomorphism if and only if $\Hom(
k^{\stab}, f)$ is a quasi-isomorphism. However, by Proposition \ref{residue.dual}
the cohomology of the complex $\Hom(k^{\stab}, X)$ is up to shift naturally
isomorphic to the cohomology of the complex $k \otimes_R X$ and the analogous
statement is true for $Y$.
\end{proof}

Finally, we point out that everything generalizes to the nonlocal situation of
\ref{subsect.nonlocal}.

\begin{cor} \label{equivalence.nonlocal} 
Let $(R,w)$ be as in \ref{subsect.nonlocal}, satisfying condition (\ref{a}).
Then all statements of Theorem \ref{equivalence} remain true verbatim.
\end{cor}

\begin{proof}
	This follows directly from Corollary \ref{global.generator}.
\end{proof}

\subsection{Kn\"orrer periodicity}
\label{knoerrer}

Let $(R,w)$ be as in \ref{subsect.nonlocal}, satisfying condition (\ref{b}). We show that the dg categories $\MF^\infty(R,w)$ and $\MF^\infty(R[x,y], w + xy)$ are weakly equivalent. This
phenomenon is known as Kn\"orrer periodicity (cf. \cite{knoerrer}) and illustrates in particular that
it is not possible to reconstruct the ambient ring $R$ from the matrix
factorization category $\MF^{\infty}(R,w)$.
Observe that if $(R,w)$ satisfies (\ref{b}), then the same is true for $(R[x,y], w + xy)$.

First, we consider the category $\MF^\infty(k[x,y], xy)$. By the results of
\ref{subsect.localization}, the stabilized residue field at the origin is a generator. It
splits into a direct sum $X \bigoplus X[1]$, where $X$ denotes the factorization
\[
\xymatrix{ 
k[x,y] \ar@<+.5ex>[r]^{x} & \ar@<+.5ex>[l]^{y} k[x,y]\text{.}
}
\]
Thus, we can as well take $X$ as a generator of the category $\MF^\infty(k[x,y],
xy)$. For example by using Lemma \ref{mcm.approx}, one immediately checks that
the endomorphism dg algebra $B$ of $X$ is weakly equivalent to the algebra $k$
concentrated in even degree.
Thus, by Corollary \ref{equivalence.nonlocal}, the category $\MF^\infty( k[x,y], xy)$ is 
weakly equivalent to the dg category of $\Zt$-graded complexes of $k$-vector
spaces.

To obtain the Kn\"orrer periodicity statement, we observe that the factorization $k^{\stab}
\otimes_k^{\Zt} (X \bigoplus X[1])$ 
in the category $\MF^{\infty}(R[x,y],w + xy)$ stabilizes the residue field at
the maximal ideal supporting the critical locus. Thus, the object $k^{\stab} \otimes_k^{\Zt} X$ 
generates the category $\MF^{\infty}(R[x,y],w + xy)$. The endomorphism dg algebra of
this generator is given by $A \otimes_k B$, where $A$ and $B$ are the
endomorphism dg algebras of $k^{\stab}$ in $\MF(R,w)$ and $X$ in $\MF(k[x,y], xy)$
respectively. But since $k$ is a field, the dg algebra $A \otimes_k^{\Zt} B$ is weakly
equivalent to $A$. In view of Corollary \ref{equivalence.nonlocal}, this implies that the dg categories 
$\MF^{\infty}(R[x,y],w + xy)$ and $\MF^{\infty}(R,w)$ are weakly equivalent. By
the results explained in \ref{formal.completion}, we also obtain an equivalence of
the categories of finite rank factorizations after passing to completions of
$R$ and $R[x,y]$.

\subsection{Formal completion}
\label{formal.completion}

Let $(R,w)$ be as in \ref{subsect.local} with isolated singular locus
(\ref{is.sing}).
We address the question of describing the triangulated hull
$\widehat{\MF(R,w)}_{\pe}$ of $\MF(R,w)$ explicitly. Let $\widehat{R}$ denote
the $\m$-adic completion of $R$ and consider the category $\MF(\widehat{R}, w)$. 

\begin{lem}\label{idem.complete} The Yoneda embedding
\[
\MF(\widehat{R}, w) \to \widehat{\MF(\widehat{R}, w)}_{\pe}
\]
is an isomorphism in $\Ho(\dgku)$. In other words, the dg category
$\MF(\widehat{R}, w)$ is triangulated and, in particular, the triangulated category $[\MF(\widehat{R}, w)]$ is
idempotent complete.
\end{lem}
\begin{proof}
	By \cite[5.3]{keller}, the category $[\widehat{\MF(\widehat{R},w)}_{\pe}]$ is the smallest triangulated subcategory of 
	$[\widehat{\MF(\widehat{R},w)}]$ which is closed under summands and contains the Yoneda image of
	$[\MF(\widehat{R},w)]$. It therefore suffices to show that
	$[\MF(\widehat{R},w)]$ is idempotent complete.

	As explained in \ref{subsect.factorizations}, the category
	$[\MF(\widehat{R},w)]$ is equivalent to the stable category associated
	to the Frobenius category $\MCM(\widehat{R}/w)$ of maximal Cohen-Macaulay modules
	over $\widehat{R}/w$. By the classical result \cite[Remark on page 566]{swan}, the endomorphism algebra $\End(M)$ of an indecomposable module $M$ over a complete local ring is local, i.e. the sum of two
	non-units is a non-unit. Now let $e$ be an element in $\End(M)$ whose
	image $\overline{e}$ in the stable endomorphism algebra
	$\underline{\End}(M)$ is a non-trivial idempotent. 
	In particular, $\overline{e}$ and $1 - \overline{e}$ are non-units.
	This certainly implies that $e$ and $1-e$ are non-units in $\End(M)$ and
	therefore $1 = e + (1-e)$ is a non-unit which is a contradiction. So
	$\underline{\End}(M)$ does not contain non-trivial idempotents which
	implies the statement.
\end{proof}
	
The lemma enables us to describe the category $\widehat{\MF(R,w)}_{\pe}$
explicitly. 

\begin{theo}\label{formal.idempotent}
	There exists an isomorphism in $\Ho(\dgku)$
	\[
	\widehat{\MF(R,w)}_{\pe} \simeq
	\MF(\widehat{R},w)\text{.}
	\]
	In particular, the idempotent completion of $[\MF(R,w)]$ in
	$[\widehat{\MF(R,w)}]$ is equivalent
	to $[\MF(\widehat{R},w)]$.
\end{theo}
\begin{proof}
	Combining Lemma \ref{idem.complete} with Theorem \ref{equivalence}, we
	obtain an isomorphism in $\Ho(\dgku)$
	\[
	\MF(\widehat{R}, w) \stackrel{\simeq}{\lra}
	\widehat{A_{(\widehat{R},w)}}_{\pe}\text{,}
	\]
	where $A_{(\widehat{R},w)}$ denotes the endomorphism dg algebra of the
	stabilized residue field in the category $\MF(\widehat{R},w)$.
	On the other hand denoting the analogous dg algebra for $\MF(R,w)$ by
	$A_{(R,w)}$ we have an isomorphism
	\[
	\widehat{\MF(R, w)}_{\pe} \stackrel{\simeq}{\lra}
	\widehat{A_{(R,w)}}_{\pe}\text{,}
	\]
	by Theorem \ref{equivalence}. There is a natural inclusion map of dg
	algebras $A_{(R,w)} \to A_{(\widehat{R},w)}$ which, using our explicit
	description of both algebras, is a quasi-isomorphism. Indeed, we can
	filter the complex underlying $A_{(R,w)}$ by the total degree in the
	variables $\partial_i$. The associated graded complex is of the
	form
	\[
	R\left<\theta_1, \ldots, \theta_n \right> \otimes_R R\left<
	\partial_1,\ldots,\partial_n \right>
	\]
	where the complex $R\left<\theta_1, \ldots, \theta_n\right>$ is the Koszul
	complex of the regular sequence $x_1, \ldots, x_n$ and the complex $R\left<
	\partial_1,\ldots,\partial_n\right>$ has zero differential. Thus, the
	cohomology is given by 
	$k\left<\partial_1,\ldots,\partial_n\right>$.
	Introducing the analogous filtration on $A_{(\widehat{R},w)}$, the
	inclusion map $A_{(R,w)} \to A_{(\widehat{R},w)}$ respects both
	filtrations and the induced map on spectral sequences is an isomorphism
	on the second page. Therefore, the inclusion map is a quasi-isomorphism.
	This weak equivalence implies an isomorphism in $\Ho(\dgku)$ between
	$\widehat{A_{(R,w)}}_{\pe}$ and $\widehat{A_{(\widehat{R},w)}}_{\pe}$.
\end{proof}

In \cite{orlov-2009}, the relation between idempotent completion and formal
completion is studied in a more general context on the level of
triangulated categories. 

\subsection{Quadratic Hypersurfaces}

Let $R = k[ [x_1,\dots,x_n]]$ and let $w \in R$ be a non-zero quadratic form. Assuming
$\operatorname{char}(k) \ne 2$ we may diagonalize $w$, so after a change of
coordinates we have $w = \sum_{i=1}^n a_i x_i^2$ with $a_i \in k$. Since we
assume the singularity to be isolated, all coefficients $a_i$ are non-zero. We write $w =
x_i w_i$ setting $w_i = a_i x_i$. With above notation we have 
\[
A = R\left<\theta_1, \ldots, \theta_n, \partial_1,
\dots, \partial_n \right>
\]
with differential $d$ given by
\begin{align*}
\theta_i & \mapsto  x_i\\
\partial_i & \mapsto a_i x_i \text{.}
\end{align*}
The elements $\overline{\partial_i} = \partial_i - a_i \theta_i$ are cycles in
$A$ and generate the cohomology $\H^*(A)$ as a $k$-algebra. Note that the relations
\begin{align*}
\overline{\partial_i}^2 &= -a_i \\ 
\overline{\partial_i}\;\overline{\partial_j} &=
-\overline{\partial_j}\;\overline{\partial_i} \quad \text{for $i \ne j$}
\end{align*}
imply that the $k$-subalgebra of $A^{\op}$ generated by
$\left\{\overline{\partial_i}\right\}$ is isomorphic to the Clifford algebra $\operatorname{Cl}(w)$ corresponding to
the quadratic form $w$. Furthermore, the inclusion 
\[
(\operatorname{Cl}(w),0) \subset (A^{\op},d)
\]
is a quasi-isomorphism establishing the formality of the dg algebra $A^{\op}$.
Therefore, in the quadratic case Theorem \ref{equivalence} reproduces a variant
of the results in \cite{buchweitz2} describing matrix factorizations as modules
over the Clifford algebra $\operatorname{Cl}(w)$ (see also \cite[Chapter
14]{yoshino}).

\subsection{A minimal $A_\infty$ model}

If $w$ is of degree greater than $2$, the algebra $A$ will
not be formal. However, there is a well-known structure which allows us
nevertheless to pass to the cohomology algebra of $A$: the structure of
an $A_\infty$ algebra. For the basic theory we refer the reader to
\cite{stasheff,chen,keller-1999,kont.soib}. The relevance to our situation is the following.
In addition to the usual multiplication on the cohomology algebra of $A$ there
exist higher multiplications. In a precise sense, they measure the failure 
of being able to choose a multiplicatively closed set of representatives of
$\H^*(A)$ in $A$. The system of higher multiplications forms an
$A_\infty$-algebra and as such $\H^*(A)$ will be quasi-isomorphic to $A$.

One of the motivations for passing from $A$ to $\H^*(A)$ is that the latter
algebra is finite dimensional over $k$; it is referred to as a \emph{minimal
model of} $A$. By the general theory in \cite{hasegawa} we
obtain a description of the category of matrix factorizations as a category of modules over the
$A_\infty$ algebra $\H^*(A)$. We do not spell out a precise formulation of
this equivalence but restrict ourselves to the description of the $A_\infty$
structure in some special cases.

We use the method described in \cite[6.4]{kont.soib2} (also cf. \cite{gug.stash, merkulov}) to compute the $A_\infty$ structure in terms of trees. 
The setup is as follows.
Let $X$ be an $A_\infty$ algebra. This structure can be described by the data of
a coderivation $q$ of degree $1$ on the coalgebra $\operatorname{T}\! X[1]$ such
that $[q,q] = 0$. The
Taylor coefficients of $q = q_1 + q_2 + q_3 + \ldots$ are maps
\[
q_k : X[1]^{\otimes k} \to X[1]
\]
which describe, after introducing the sign shifts accounting for the transfer
from $X[1]$ to $X$, the higher multiplications $m_k$ on $X$.
Now assume, that
we are given an idempotent $p : X \to X$ of degree $0$ commuting with $d$. The
image of $p$ is therefore a subcomplex of $X$ which we denote by $Y$. Let $i$
denote the inclusion of $Y$ into $X$. We also
assume, that a homotopy $h : X \to X[1]$ between $\id$ and $p$ is given, i.e.
\[
dh + hd = \id - p
\]
identifying $p$ as a homotopy equivalence between $X$ and $Y$. In this
situation, the $A_\infty$-structure induced on $Y$ can be calculated explicitly
in terms of trees as for example described in \cite[6.4]{kont.soib2}.
This method can in principle be used to determine the $A_\infty$ structure
on $\H^*(A)$ for $R = k[ [x_1,\dots,x_n]]$. However the computation is rather tedious and the
author has not been able to find closed formulas for all higher multiplications.
Since we will not use the results of this calculation elsewhere, we will 
define the contracting homotopy $h$ and only state the results of the tree
calculation.

We start with the one-dimensional case $R = k[ [x]]$ which
will serve as a guideline for what to do in the higher dimensional case.
For an element $a$ in $k[ [x]]$ we define $\overline{a}$ to be the unique
element in $x k[ [x]]$ defined by
\[
a = a_0 + \overline{a}
\] 
with $a_0 \in k$.
The algebra $A$ is of the form
\[
A \cong k[ [x]] \otimes_k k\left< \theta, \partial \right>
\]
with differential $d$ given by
\begin{align*}
\theta & \mapsto x\\
\partial & \mapsto \frac{w}{x} \text{.}
\end{align*}
We define a homotopy $h$ contracting $A$ onto its cohomology. The assignment
$h(a) = \frac{\overline{a}}{x}\theta$ for $a \in k[ [x]]$ extends to a unique
$k\left< \theta,\partial \right>$-linear homotopy of $A$.
A simple calculation shows that the map
\[
p = \id - [d, h]
\]
is a projection and we have maps of complexes
\[
\xymatrix{
k[ [x]] \otimes_k k\left< \theta, \partial \right>
\ar@<+.5ex>[d]^p \\ \ar@<+.5ex>[u]^{\iota} k\left<\overline{\partial}\right>
}
\]
where the differential on $k\left<\overline{\partial}\right>$ is $0$ and 
$\iota(\overline{\partial}) = \partial - \frac{w}{x^2}\theta$.

We have determined all the data needed for the tree formula. It leads to the
following higher multiplications.

\begin{theo} Let $R = k[ [x]]$ with $w = \sum_{i=2}^{\infty} r_i
x^i$. Then the unital $A_\infty$-structure on $\H^*(A) \cong k 1 \oplus k
\overline{\partial}$ induced by the above homotopy is uniquely determined by the formulas
\[
m_i(\overline{\partial}, \dots, \overline{\partial}) = \pm r_i  \text{.}
\]
\end{theo}

In the more general case $R = k[ [x_1,\dots,x_n]]$ one can still construct an
explicit contracting homotopy, however the resulting formulas get more
complicated. The following result does not determine the $A_\infty$ structure
completely but at least shows some interesting properties.

\begin{theo} \label{Aformula} Let $R = k[ [x_1, \dots, x_n]]$ with 
$w = \sum_{\underline{i}} r_{\underline{i}} x^{\underline{i}}$ where
$\underline{i} = (i_1,\dots, i_n)$ denotes a multi-index. Then there exists a
contracting homotopy of $A$ such that the induced $A_\infty$ structure on
$\H^*(A)$ has the following properties
\begin{itemize}
\item The underlying associative algebra on $\H^*(A)$ is given by the Clifford
algebra corresponding to the quadratic term of $w$ as described in the previous
subsection.
\item For the generators $\overline{\partial_1}, \dots, \overline{\partial_n}$
we have 
\[
m_i(\underbrace{\overline{\partial_1},\dots,\overline{\partial_1}}_{i_1},\underbrace{\overline{\partial_2},\dots,
\overline{\partial_2}}_{i_2},\dots,\underbrace{\overline{\partial_n},\dots,
\overline{\partial_n}}_{i_n})
= \pm r_{\underline{i}}
\]
where $i = |\underline{i}|$.
\end{itemize}
\end{theo}

Note that the formulas on the generators $\overline{\partial}_i$ do not
determine the $A_\infty$-structure. Still, one instructive feature is the direct relation to the coefficients of $w$. Formality can
thereafter only be expected if $w$ is quadratic. Further, the
formulas raise the question of how much information about $w$ can be recovered
from the weak equivalence class of the dg algebra $A$. This question is explored
in \cite{efimov} by using Kontsevich formality.

We illustrate the determination of the contracting homotopy for $R = k[ [x_1, x_2]]$. Note that it is possible
to write $w = x_1 w_1 + x_2 w_2$ with $w_1 \in k[ [x_1]]$ and $w_2 \in k[ [x_1,
x_2]]$. Then the endomorphism algebra $A$ of the stabilized residue field is of the form
\[
k[ [x_1]]\left< \theta_1, \partial_1 \right> \otimes_{k[ [x_1]]} k[ [x_1,
x_2]]\left< \theta_2, \partial_2 \right>
\]
with differential $d$ given by
\begin{align*}
\theta_i \mapsto x_i\\
\partial_i \mapsto w_i
\end{align*}
Applying the construction from the one dimensional case twice, we obtain maps of
complexes
\[
\xymatrix@R=0.5ex{
k[ [x_1]] \left<\theta_1, \partial_1 \right> \otimes_{k[ [x_1]]} k[ [x_1,
x_2]]\left<\theta_2, \partial_2\right> \ar@<+.5ex>[dd]^{p_2} \\ 
& \iota_2 \circ p_2 = id - [d,h_2] \\
k[ [x_1]] \left<\theta_1, \partial_1 \right> \otimes_{k[ [x_1]]} k[ [x_1]]\left<
\widetilde{\partial_2}\right> \ar@<+.5ex>[dd]^{p_1} \ar@<+.5ex>[uu]^{\iota_2}\\
& \iota_1 \circ p_1 = id - [d,h_1] \\
\ar@<+.5ex>[uu]^{\iota_1}   k<\overline{\partial_1}, \overline{\partial_2}>
}
\]
where $h_2(a) = \frac{a}{x_2} \theta_2$ for $a \in k[ [x_1,x_2]]$ and $h_1(b) =
\frac{b}{x_1}\theta_1$ for $b \in k[ [x_1]]$. Here, the symbol $\frac{}{y}$ denotes division by $y$, discarding the remainder.
As in the one
dimensional case, $h_1$ and $h_2$ are extended linearly to yield the homotopies
in the above diagram. 
In fact, we can collapse the sequence of homotopy equivalences to a single one
given by
\begin{align*}
\iota & =  \iota_2 \circ \iota_1\\
p &= p_1 \circ p_2\\
h &= h_2 + \iota_2 \circ h_1 \circ p_2
\end{align*}
The inclusions into $A$ are given by
\begin{align*}
\overline{\partial_1} & = \partial_1 - \frac{w_1}{x_1}\theta_1\\
\overline{\partial_2} & = \partial_2 - \frac{w_2}{x_2}\theta_2 -
\frac{r_2}{x_1}\theta_1
\end{align*}
where $r_2 \in k[ [x_1]]$ is the remainder of the division of $w_2$ by $x_2$.

An application of the tree formula yields the formulas described in Theorem
\ref{Aformula}.
This calculation generalizes to an arbitrary number of variables.\\

\section{Derived Morita theory}
\label{sect.morita}

In this section we will use the operation
\[
(R,w) \otimes_k (R',w') = (R \otimes_k R', w \otimes 1 + 1 \otimes w') \text{.}
\]
In order for it to be well-behaved, we have to impose conditions on
$(R,w)$. We require $(R,w)$ and $(R',w')$ to satisfy condition (\ref{b}) from \ref{subsect.nonlocal}.
These assumptions on $(R,w)$ and $(R',w')$ ensure that $(R \otimes_k R', w
\otimes 1 + 1 \otimes w')$ will still satisfy condition $(\ref{b})$, making the
localization results of \ref{subsect.localization} and Corollary
\ref{equivalence.nonlocal} applicable. Also note, that due to
\ref{subsect.localization}, we can without loss of generality assume that $R$
and $R'$ are local. 

As proved in \cite[Section 6]{toen.morita} the category $\Ho(\dgku)$ admits internal
homomorphism categories satisfying the usual adjunction
\[
[U \otimes^{\L} T, T'] \cong [ U, \IRHom( T, T') ] \text{.}
\]
We will use Morita theory 
\cite[Section 7]{toen.morita} to determine the dg category of functors between
two matrix factorization categories.
As an application, we calculate the Hochschild chain and cochain complexes of these categories.

\subsection{Internal homomorphism categories}
Given two categories $\MF^\infty(R,w)$ and
$\MF^\infty(R',w')$ there is a natural class of dg functors between them. Namely,
every object $T$ in the category $\MF^\infty(R \otimes_k R', - w \otimes
1+ 1\otimes w')$ defines a dg functor via the association
\[
\MF^\infty(R,w) \to \MF^\infty(R',w'), \; X \mapsto X
\otimes_R T \text{,}
\]
where the tensor product is $\Zt$-graded.
In other words, the object $T$ acts as the kernel of an integral
transform. We will show that every continuous functor between matrix factorization
categories is isomorphic to an integral transform.

By Theorem \ref{equivalence}, there is an
isomorphism
\[
\MF^{\infty}(R,w) \stackrel{\simeq}{\lra} \widehat{A}
\]
where $A$ is the $2$-periodic endomorphism dg algebra of the compact generator $E = k^{\stab}$.
The matrix factorization
$E^{\vee} = \Hom_R^{\Zt}(E,R)$
is a compact generator of the category $\MF^\infty(R,-w)$. Indeed, as already
pointed out in the proof of Lemma \ref{residue.dual}, one explicitly verifies that
$E^{\vee}$ stabilizes a shift of the residue field. Therefore, 
Theorem \ref{equivalence} gives a natural isomorphism
\[
\MF^{\infty}(R,-w) \stackrel{\simeq}{\lra} \widehat{A^{\op}} \text{.}
\]

Consider a second hypersurface singularity $(R',w')$ with corresponding category
$\MF^\infty(R',w')$ and stabilized residue field $E'$. 
One immediately identifies the object $E \otimes_k
E'$ of the category $\MF^\infty(R\otimes_k R', w \otimes 1 + 1
\otimes w')$ as a stabilization of the residue field. Therefore, we obtain an
isomorphism of dg algebras
\[
A_{(R\otimes_k R',w \otimes 1 + 1 \otimes w')} \stackrel{\cong}{\lra} A_{(R,w)}
\otimes A_{(R',w')}\text{.}
\]

Combining Theorem 1.4 in \cite{toen.morita} with the above observations and Corollary \ref{equivalence.nonlocal},
we obtain the following result.

\begin{theo} There exists a natural isomorphism in $\Ho(\dgku)$ 
\[
\IRHom_c( \MF^\infty(R,w), \MF^\infty(R',w')) \cong \MF^\infty(R \otimes_k R', - w
\otimes 1 + 1 \otimes w') \text{.}
\]
\end{theo}

We conclude the subsection with a compatibility statement.

\begin{prop}
	Let $T$ be an object in $\MF^\infty(R \otimes_k R', - w \otimes 1 + 1
	\otimes w')$, let $E$ resp. $E'$ be the compact generators of
	$\MF^\infty(R,w)$ resp. $\MF^\infty(R',w')$ as constructed above. Then
	the diagram of functors
	\[
	\xymatrix@C=3cm{
	\left[ \MF^\infty(R,w) \right] \ar[d]_{\Hom(E,-)} \ar[r]^{- \otimes_R T} &
	\ar[d]^{\Hom(E',-)} \left[\MF^\infty(R',w') \right]\\
	\D(A^{\op}) \ar[r]^{- \otimes^{\operatorname{L}}_A \Hom(E^{\vee}
	\otimes_k E', T)} & \D(A'^{\op})
	}
	\]
	commutes up to a natural equivalence.
\end{prop}

\begin{proof}
	Using the natural isomorphism of complexes
	\[
	\Hom_{R \otimes_k R'} (E^{\vee} \otimes_k E', T) \cong \Hom_{R'}(E', E \otimes_R T)
	\]
	we obtain a natural transformation
	\[
		\Hom_R(E,-) \otimes^{\operatorname{L}}_A \Hom_{R'}(E^{\vee}
		\otimes_k E', T) \to \Hom(E', - \otimes_R T)
	\]
	via composition. Both functors respect the triangulated structure
	and commute with infinite coproducts. Evaluated on the compact generator
	$E$, the above transformation yields an isomorphism in $\D( A'^{\op})$. Therefore, it must
	be an equivalence of functors on $\left[ \MF^\infty(R,w) \right]$.
\end{proof}

\subsection{Hochschild cohomology}

One of the many neat applications of the homotopy theory 
developed in \cite{toen.morita} is the description of the Hochschild cochain
complex of a dg category as the endomorphism complex of the identity functor.
This result carries over to the $2$-periodic case and we will use it to
determine the Hochschild cohomology of the dg category $\MF^\infty(R,w)$.

Consider the matrix factorization category $\MF^\infty(R,w)$ corresponding to an isolated
hypersurface singularity. We choose the stabilized residue field as a compact
generator which yields an isomorphism
\[
\MF^{\infty}(R,w) \stackrel{\simeq}{\lra} \widehat{A} 
\]
in $\Ho(\dgku)$ as explained above.
Let us introduce the notation $\widetilde{w} = -w\otimes 1 + 1 \otimes w$.
We have to identify an object in $\MF^\infty(R\otimes_k R,\widetilde{w})$ which induces the identity functor on $\MF^\infty(R,w)$. Equivalently, we have to find an object whose image under the equivalence
\[
\left[ \MF^\infty(R\otimes_k R,\widetilde{w}) \right] \stackrel{\simeq}{\lra}
\D(A \otimes A^{\op})
\]
given by $\Hom(E^{\vee} \otimes_k E, -)$ is isomorphic to the $A \otimes A^{\op}$-module $A$. 
There is an obvious candidate for
the integral kernel which induces the identity functor: the \emph{stabilized diagonal}
$\Delta^{\stab}$. Analogously to the stabilized residue field, it is defined 
as a stabilization of $R$ considered as an $R\otimes_k R /
\widetilde{w}$-module. 
To prove that $\Delta^{\stab}$ is actually isomorphic to the identity functor we
will use Lemma \ref{mcm.approx}.

\begin{prop}\label{diag}
	The stabilized diagonal $\Delta^{\stab}$ is isomorphic to the identity
	functor on $\MF^\infty(R,w)$.
\end{prop}

\begin{proof}
	We apply Lemma \ref{mcm.approx} with 
	\[
	X = E^{\vee} \otimes_k E 
	\]
	where $E$ is the stabilized residue field in $\MF(R,w)$ and $Y = \Delta^{\stab}$ which is the
	stabilization of the diagonal $R$ as an $R\otimes_k R/\widetilde{w}$-module. 
	We obtain an isomorphism
	\[
	\MF(E^{\vee} \otimes_k
	E, \Delta^{\stab}) \cong \Hom_{R\otimes_k
	R}^{\Zt}(E^{\vee} \otimes_k E, R)
	\]
	and further
	\begin{align*}
		\Hom_{R\otimes_k R}^{\Zt}(E^{\vee} \otimes_k E, R)
		& \cong \Hom_R^{\Zt}(E^{\vee} \otimes_R E, R)\\
		& \cong \Hom_R^{\Zt}(E, E)\\ &\cong \MF(E,E)
		\text{.}
	\end{align*}
	But the latter complex is by definition $A$ and all the maps respect the
	$A \otimes A^{\op}$-module structure. 
\end{proof}

We can make the stabilized diagonal explicit after passing to the
formal completion. Indeed, by combining Theorem \ref{equiv.localization}
and Theorem \ref{formal.idempotent}, we obtain a weak equivalence
\[
\MF^\infty(R\otimes_k R,\widetilde{w}) \stackrel{\simeq}{\lra} \MF^\infty(\widehat{R\otimes_k
R},\widetilde{w})\text{,}
\]
where $\widehat{R \otimes_k R}$ denotes the completion with respect to the
maximal ideal supporting the critical locus of $\widetilde{w}$. Over $\widehat{R
\otimes_k R}$, the stabilized diagonal admits an explicit description 
via the Koszul method described in subsection \ref{stabilization}. 
Indeed, for a minimal system of generators $t_1, \dots, t_n$ of the maximal
ideal $\m \subset R$,
we obtain an isomorphism 
\[
\widehat{R \otimes_k R} \cong k[ [ x_1, \dots, x_n, y_1, \dots, y_n] ] \text{,}
\] 
where $x_i = t_i \otimes 1$ and $y_i = 1 \otimes t_i$. In this situation, the
diagonal $\widehat{R}$ is given as the quotient $\widehat{R \otimes_k R} / I$
where $I = (\Delta_1, \dots, \Delta_n)$ with $\Delta_i = x_i - y_i$.
We can find an expression of the form
$\widetilde{w} = \sum_{i=1}^n \Delta_i \widetilde{w}_i$. 
By Corollary \ref{cor.stab}, the matrix factorization
$\Delta^{\stab}_{\widehat{R\otimes_k R}}$ in $\MF^\infty(\widehat{R\otimes_k R},\widetilde{w})$ is of the form
\[
(\bigoplus_{i=0}^n {\textstyle\bigwedge\nolimits}^{i} V, \;s_0 + s_1 ) \text{,}
\]
with $s_0$ given by contraction with $(\Delta_1, \dots, \Delta_n)$ and $s_1$ by
exterior multiplication with $(\widetilde{w}_1,\dots,\widetilde{w}_n)^{\tr}$. We
abbreviate the factorization by $\Delta^{\stab}_{\widehat{R\otimes_k R}}$ by
$\widehat{\Delta}^{\stab}$.

\begin{cor} \label{jacobi} The Hochschild cochain complex of $\MF^\infty(R,w)$ is quasi-isomorphic to 
     the $\Zt$-folded Koszul complex of the regular sequence $\partial_1 w,
     \dots, \partial_n w$ in $R$. In particular, the Hochschild cohomology is
     isomorphic, as an algebra, to the Jacobian algebra
\begin{align*}
\HH^*(\MF^{\infty}(R,w)) & \cong R / (\partial_1 w, \dots, \partial_n w )
\end{align*}
     concentrated in even degree.
\end{cor}

\begin{proof}
	By \cite[Corollary 8.1]{toen.morita} and Proposition \ref{diag} the Hochschild cochain complex is
	quasi-isomorphic to
	\[
	\MF(\Delta^{\stab},\Delta^{\stab})
	\text{.}
	\]
	By the discussion preceding the Corollary, we can pass to the formal
	completion $\widehat{R \otimes_k R}$ and use the explicit Koszul
	description of $\widehat{\Delta}^{\stab}$.
	We apply Lemma \ref{mcm.approx} with $X = Y = \widehat{\Delta}^{\stab}$. Since 
	$Y$ stabilizes the $\widehat{ R\otimes_k R} /\widetilde{w}$-module
	$\widehat{R}$, we
	obtain an isomorphism
	\[
	\MF(\widehat{\Delta}^{\stab}, \widehat{\Delta}^{\stab}) \cong \Hom_{
	\widehat{R\otimes_k R}}^{\Zt}(\widehat{\Delta}^{\stab}, \widehat{R}) \text{.}
	\]
	The latter complex is isomorphic to the Koszul complex of
	the sequence formed by the reduction of the elements 
	$\widetilde{w_1}, \dots, \widetilde{w_n}$ modulo the ideal $(\Delta_1,
	\dots, \Delta_n)$.
	We only have to observe that 
	$\widetilde{w_i}$ is congruent to $\partial_i w$ modulo $(\Delta_1, \dots,
	\Delta_n)$. Indeed, 
	\begin{align*}
		\partial_i w(x) & = \lim_{\Delta_i \to 0} \frac{w(x + \Delta_i) -
		w(x)}{\Delta_i}\\
		& = \lim_{\Delta_i \to 0} \frac{\widetilde{w}\; \mod\; (\Delta_1,\dots,
		\widehat{\Delta_i}, \dots, \Delta_n)}{\Delta_i}\\
		& = \lim_{\Delta_i \to 0} \widetilde{w_i}\; \mod \; (\Delta_1,\dots,
		\widehat{\Delta_i}, \dots, \Delta_n)\\
		& = \widetilde{w_i} \; \mod \; (\Delta_1,\dots,\Delta_n)
	\end{align*}

	Finally, one verifies that the $0$-cycle in
	$\MF(\widehat{\Delta}^{\stab},\widehat{\Delta}^{\stab})$ given by multiplication by a scalar
	of the form $r\otimes 1$ maps under the quasi-isomorphism of Proposition \ref{exp.approx} 
	to a cycle which represents the residue class of $r$ in the Jacobian
	algebra. Thus, the induced map on cohomology is an algebra homomorphism.
\end{proof}

From this calculation we can also obtain the Hochschild cohomology of the finite
rank category $\MF(R,w)$.

\begin{cor} The Hochschild cohomology of the $2$-periodic dg category $\MF(R,w)$
	is isomorphic to the Jacobian algebra
	\begin{align*}
\HH^*(\MF(R,w)) & \cong R / (\partial_1 w, \dots, \partial_n w )
\end{align*}
concentrated in even degree.
\end{cor}
\begin{proof}
By Theorem \ref{equivalence} we have an isomorphism $\MF^{\infty}(R,w) \simeq
\widehat{\MF(R,w)}$ in the category $\Ho(\dgku)$. Therefore, the statement follows
immediately from \cite[Corollary 8.2]{toen.morita}.
\end{proof}

Note that, by the same argument, $\HH^*(\widehat{\MF(R,w)}_{\pe})$ is
isomorphic to the Jacobian algebra.

\subsection{Hochschild homology}

We draw attention to the well-known fact that the category
$[\MF(R,w)]$ is a Calabi-Yau category (cf. \cite[10.1.5]{buchweitz}). A lift of
this result to a statement about dg categories would therefore show that
Hochschild cochain and chain complex are in duality via the trace pairing. 
According to the above computation we would then expect the following result to hold
for the Hochschild homology of the category $\MF(R,w)$. 

\begin{theo} \label{hochschild} The Hochschild homology of the $2$-periodic dg category $\MF(R,w)$
	is given by
	\begin{align*}
		\HH_*(\MF(R,w)) & \cong R / (\partial_1 w, \dots, \partial_n w )
	\end{align*}
	concentrated in the degree given by the parity of the Krull dimension of $R$.
\end{theo}

We give a proof of this theorem which does not refer to the trace pairing.
Along the way, we actually prove that matrix factorization categories are Calabi-Yau in the sense of \cite[4.28]{pantev}.

The following definition of Hochschild homology is due to To\"en and we
reformulate it in the $2$-periodic situation. Let $T$ be a
$2$-periodic dg category. Let $\bone$ denote $\ku$ considered as a dg category with a single
object. Applying \cite[Lemma 6.2]{toen.morita} we obtain an isomorphism
\[
[ T \otimes T^{\op}, \widehat{\bone}] \cong \Iso( \Ho( \mmod{T \otimes
T^{\op}}))\text{,}
\]
where $\Iso$ refers to the set of isomorphism classes of objects.
On the other hand, by \cite[Theorem 7.2]{toen.morita} we have a natural
isomorphism
\[
[ \widehat{T \otimes T^{\op}}, \widehat{\bone}]_c \stackrel{\simeq}{\lra} [ T
\otimes T^{\op}, \widehat{\bone}]
\]
given by the pullback functor.
Therefore $T$, considered as an object in $\mmod{T \otimes T^{\op}}$ gives rise
to a continuous functor $\widehat{T \otimes T^{\op}} \to \widehat{\bone}$. Passing to homotopy
categories we obtain a map of derived categories
\[
\tr:\; \D(T^{\op} \otimes T) \to \D(\ku) \text{.}
\]
The Hochschild chain complex of $T$ is then defined to be the image of $T$ under this
map which, in Morita theoretic terms, coincides with the trace of the identity
functor on $T$.

If $A$ is a $2$-periodic dg algebra, then there is an alternative description of
the Hochschild chain complex of $A$. Let us
introduce the notation $A^e = A \otimes A^{\op}$. Consider the tensor product
\[
\mmod{(A^e)^{\op}} \times \mmod{A^e} \to \C(\ku),\; (M,N) \mapsto M
\otimes_{A^e} N
\]
which is defined to be the coequalizer of the two natural maps
\[
\xymatrix{
M \otimes A^e \otimes N \ar@<+.5ex>[r]\ar@<-.5ex>[r] & M \otimes N \text{.}
}
\]
The tensor product is a Quillen bifunctor and can thus be derived.
In this situation, one checks directly from the definition that the map 
\[
\tr:\; \D(A^{\op} \otimes A) \to \D(\ku)
\]
is given by the functor $-\otimes_{A^e}^{\L} A$. Therefore, the Hochschild
chain complex of $A$ admits the familiar description
\[
\CC_*(A) = A \otimes_{A^e}^{\L} A \text{.}
\]
The following lemma is well-known.

\begin{lem}\label{morita.invariance} The Hochschild chain complex of a $2$-periodic dg category $T$ and
its triangulated hull $\widehat{T}_{\pe}$ are isomorphic in $\D(\ku)$.
\end{lem}
\begin{proof}
Consider the natural functor $f:\; T \otimes T^{\op} \to \widehat{T}_{\pe}
\otimes (\widehat{T}_{\pe})^{\op}$. 
By the dg Yoneda lemma, the restriction of the $\widehat{T}_{\pe}
\otimes (\widehat{T}_{\pe})^{\op}$-module $\widehat{T}_{\pe}$ along $f$ coincides
with $T$. From this, we obtain a commutative diagram
\[
\xymatrix{
\widehat{T}_{\pe} \otimes (\widehat{T}_{\pe})^{\op} \ar[r]
& \widehat{\widehat{T}_{\pe} \otimes (\widehat{T}_{\pe})^{\op}} \ar[r] &
\widehat{\bone}\\
T \otimes T^{\op} \ar[r] \ar[u]^f & \widehat{T \otimes T^{\op}} \ar[u]^{f_!}
\ar[r] & \widehat{\bone}\ar[u]^{\id}
}
\]
in $\Ho(\dgku)$, where the horizontal functors are the ones constructed in the
definition of the Hochschild chain complex above.
To obtain the result, we have to show that the functor 
\[
[f_!]:\; \D(T \otimes T^{\op}) \to \D(\widehat{T}_{\pe} \otimes
(\widehat{T}_{\pe})^{\op})
\]
maps $T$ to $\widehat{T}_{\pe}$. Since, by iterated application of \cite[Lemma 7.5]{toen.morita}, we have a Quillen equivalence
\[
\xymatrix@M=0.2cm{
\mmod{T \otimes T^{\op}} \ar@<+.5ex>[r]^(0.45){f_!} & \ar@<+.5ex>[l]^(0.55){f^*} \mmod{\widehat{T}_{\pe} \otimes
(\widehat{T}_{\pe})^{\op}}
}
\]
it suffices to show that $f^*$ maps $\widehat{T}_{\pe}$ to $T$. This follows from the dg Yoneda lemma.
\end{proof}

Let $E$ be the stabilized residue field in the category $\MF(R,w)$ and denote
$\Hom(E,E)$ by $A$.
We have morphisms in $\Ho(\dgku)$
\[
\MF(R,w) \lra \widehat{\MF(R,w)}_{\pe} \stackrel{\simeq}{\lra} \widehat{A}_{\pe}
\longleftarrow A
\]
where the middle morphism is the isomorphism from Theorem \ref{equivalence}. 
Applying Lemma \ref{morita.invariance}, we conclude that the Hochschild chain complexes of the categories
$\MF(R,w)$ and $A$ are isomorphic in $\Ho(\C(\ku))$. 
Then, using that $A$ is a perfect
$A^e$-module by Proposition \ref{diag}, we have an isomorphism
\[
\CC_*(A)  \simeq A \otimes^{\text{L}}_{A^e} A \simeq \RHom_{A^e}( A^!, A)
\]
in $\Ho(\dgku)$. Here, $\RHom_{A^e}$ denotes the $\Ho(\C(\ku))$-enriched derived Hom functor
with respect to the natural $\C(\ku)$-module structure on the model
category $\mmod{A^e}$
and we define
\[
A^! = \RHom_{(A^e)^{\op}}(A, A^e) \text{.}
\]
Note that $A^!$ admits a natural $A^e$-module
structure.

Via the compact generator $E \otimes_k E^{\vee}$ we obtain an isomorphism 
\[
\MF^\infty(R\otimes_k R,-\widetilde{w}) \stackrel{\simeq}{\lra}
\widehat{(A^e)^{\op}} =
\Int( \mmod{A^e}) \text{.}
\]
Therefore, we can calculate $\CC_*(A)$ as a morphism complex in
$\MF^\infty(R\otimes_k R,-\widetilde{w})$, provided we determine the matrix
factorization corresponding to $A^!$.

\begin{lem} \label{calabiyau} The matrix factorization corresponding to $A^!$ is the stabilized
diagonal shifted by the parity of the dimension of $R$.
\end{lem}

\begin{proof} Let $\widetilde{E} = E \otimes_k E^{\vee}$. We have to find a matrix
factorization $X$ in $\MF(R\otimes_k R, -\widetilde{w})$ such that
\[
\MF(\widetilde{E}, X) \simeq \RHom_{(A^e)^{\op}}(A, A^{e})
\]
For any factorization $X$, we have
\begin{align*}
\MF(\widetilde{E}, X) & \cong \MF(X^{\vee}, \widetilde{E}^{\vee}) \text{,}
\end{align*}
where the right-hand side is a morphism complex in the category $\MF(R\otimes_k R,
\widetilde{w})$.
Now $\widetilde{E}^{\vee}$ is a compact generator of $\MF^{\infty}(R\otimes_k R,
\widetilde{w})$ with endomorphism dg algebra $A^e$. By Theorem
\ref{equivalence}, we obtain a quasi-isomorphism
\begin{align*}
\MF(X^{\vee}, \widetilde{E}^{\vee}) 
\simeq \RHom_{(A^e)^{\op}}(\MF(\widetilde{E}^{\vee},X^{\vee}), A^e) \text{.}
\end{align*}
We choose $X$ to be the stabilized diagonal shifted by the parity of the
dimension of $R$. Then $X^{\vee}$ is isomorphic to the stabilized diagonal
in the category $\MF(R\otimes_k R,\widetilde{w})$. This can be explicitly
verified by passing to completions and using the explicit Koszul description of
the stabilized diagonal (cf. discussion preceding Corollary \ref{jacobi}). Finally, the argument of Proposition \ref{diag} yields a quasi-isomorphism
\[
\MF(\widetilde{E}^{\vee},X^{\vee}) \simeq A
\]
completing the proof.
\end{proof}

Note that the lemma in combination with Corollary \ref{jacobi} immediately
implies Theorem \ref{hochschild}. We also remark that, in light of derived Morita
theory, the bimodule $A^!$ determines an endofunctor on the category
$\MF(R,w)$. It corresponds to the inverse Serre functor and the fact that
it is isomorphic to a shift of the identity expresses the Calabi-Yau property of
$\MF(R,w)$.

\section{Noncommutative geometry}
\label{sect.geom}

We conclude with some remarks on the geometric implications of our results.
In the introduction, we pointed out that we want to think of the dg category of
matrix factorizations as the derived category of sheaves on a noncommutative space. 
Let $\X$ be a hypothetical noncommutative space associated to an isolated hypersurface 
singularity $(R,w)$, where $R$ is a regular local $k$-algebra with residue
field $k$. Without explicitly knowing how to think of $\X$ itself, we
postulate
\[
\D^\text{qcoh}_{\text{dg}}(\X) \simeq \MF^{\infty}(R,w) \text{.}
\]
The symbol $\X$ is thus merely of linguistical character, the defining mathematical
structure attached to it is $\D^\text{qcoh}_{\text{dg}}(\X)$. 
We show how to establish several important
properties of the space $\X$, using the results of the previous sections. Details
on the terminology which we use can be found in \cite{kontsevich-2006} and
\cite{pantev}.\\ 

\noindent
{\bf $\X$ is dg affine.} A noncommutative space $\X$ is called dg affine if
$\D^\text{qcoh}_{\text{dg}}(\X)$ is quasi-equivalent to the dg derived category
of some dg algebra $A$.
In our case, this is expressed by Theorem \ref{equivalence} where $A$ is
explicitly given as the endomorphism algebra of the stabilized residue field.\\

\noindent
{\bf Perfect complexes on $\X$.}
A general noncommutative space $\X$ is given by the dg category 
$\D^\text{qcoh}_{\text{dg}}(\X)$ of quasi-coherent sheaves on $\X$. The
category $\D^\text{perf}_{\text{dg}}(\X)$ is then defined to be the
full dg subcategory of compact objects.
In the case at hand, we have 
\[
\D^\text{perf}_{\text{dg}}(\X) \simeq \widehat{\MF(R,w)}_{\pe} 
\]
by Corollary \ref{split}. Note that the explicit description
\[
\widehat{\MF(R,w)}_{\pe} \simeq \MF(\widehat{R},w)
\]
given in Theorem \ref{formal.idempotent} implies that $\X$ in fact only depends
on the formal germ of $(R,w)$.\\

\noindent 
{\bf $\X$ is proper over $k$.}
By definition, a dg affine noncommutative space is proper over $k$ if the
cohomology of its defining dg algebra is finite dimensional over $k$. In our
case, the algebra $\H^*(A)$ is isomorphic to a finitely generated Clifford algebra and thus
finite dimensional. Note that, in our $2$-periodic situation, the finiteness condition refers
to the $\Zt$-graded object $A$. To be more precise, we should call $\X$ proper over
$\ku$.\\

\noindent 
{\bf $\X$ is homologically smooth over $k$.}
A dg affine noncommutative space is defined to be homologically smooth if $A$ is a perfect
$A \otimes A^{\op}$-module. For $\X$ this follows from the proof of Proposition
\ref{diag}: the stabilized diagonal is a compact object in
$[ \MF^{\infty}(R\otimes_k R, \widetilde{w}) ]$ which
maps to $A$ under the coproduct preserving equivalence
\[
[\MF^{\infty}(R\otimes_k R, \widetilde{w})] \stackrel{\simeq}{\lra} \D(A \otimes
A^{\op}) \text{.}
\]
The homological smoothness of $\X$ suggests that we may as well think of the category 
$\D^\text{perf}_{\text{dg}}(\X)$ as an analogue of the bounded derived category of
coherent sheaves on $\X$.\\

\noindent 
{\bf Hodge-to-de Rham degeneration for $\X$.}
The Hochschild homology of the category $\D^\text{perf}_{\text{dg}}(\X)$ should be thought of as the
Hodge cohomology of the space $\X$. The periodic cyclic homology of
$\D^\text{perf}_{\text{dg}}(\X)$ plays the role of the de Rham cohomology of $\X$. 
Generalizing the case of a commutative scheme,
there is a spectral sequence from Hochschild homology to periodic cyclic
homology. For the space $\X$ this spectral sequence degenerates, confirming
the general degeneration conjecture in the case of matrix factorization categories. Indeed, this
immediately follows from the fact that the Hochschild homology is concentrated
in a single degree. Therefore, Connes' $B$ operator must vanish on all higher pages of
the Hodge-to-de Rham spectral sequence since
it has degree $1$. In particular, we obtain 
\[
\HP_*(\MF(R,w)) \cong \HH_*(\MF(R,w)) 
\]
where $\HP_*$ denotes periodic cyclic homology.\\

\noindent 
{\bf $\X$ is a Calabi-Yau space.}
Lemma \ref{calabiyau} implies the existence of an isomorphism 
\[
A^! \simeq A[n]
\]
in $\D(A \otimes A^{\op})$ where $n$ is the dimension of $R$. Thus,
$\X$ is a Calabi-Yau space in the sense of \cite[4.28]{pantev}. In view of
\cite[5.2]{dyckmurf} this gives, in the case of matrix factorization categories, an affirmative answer to a general conjecture by Kontsevich-Soibelman \cite[11.2.8]{kontsevich-2006}. 

\bibliographystyle{halpha}
\bibliography{biblio}

\end{document}